\documentclass[10pt]{amsart}
\usepackage{geometry,color}                
\usepackage{graphicx}
\usepackage{amssymb}
\usepackage{epstopdf}
\usepackage[colorlinks=true, pdfstartview=FitV, linkcolor=blue, citecolor=blue, urlcolor=blue]{hyperref}

\usepackage{amsmath,amsthm,amsfonts, amssymb,eucal}

\theoremstyle{plain}
\newtheorem{thmns}{Theorem}
\newtheorem{lemns}[thmns]{Lemma}
\newtheorem{corns}[thmns]{Corollary}
\newtheorem{propns}[thmns]{Proposition}

\DeclareMathOperator{\Real}{Re}

\DeclareMathOperator{\interior}{Int}
\DeclareMathOperator{\exterior}{Ext}
\DeclareMathOperator{\Graph}{Graph}

\newcommand{\abs}[1]{\lvert#1\rvert}				
\newcommand{\inv}{^{-1}}							
\DeclareMathOperator{\der}{d}
\newcommand{\ds}{\, \der\! s}						
\newcommand{\dt}{\, \der\! t}						
\newcommand{\dm}{\der\! m}						
\newcommand{\dmu}{\der\! \mu}					
\newcommand{\domega}{\der\! \omega}				
\newcommand{\pardrv}[2]{\frac{\partial #1}{\partial #2}}	%partial deriv
\newcommand{\ra}{\ensuremath{\rightarrow}}
\newcommand{\of}{\circ}							%composition
\newcommand{\intT}{\ensuremath{\int_\T}}

\newcommand{\al}{\ensuremath{\alpha}}				% alpha
\newcommand{\lb}{\ensuremath{\lambda}}				% lambda
\newcommand{\ph}{\ensuremath{\varphi}}				% phi
\newcommand{\Ph}{\ensuremath{\Phi}}				% cap Phi
\newcommand{\ps}{\ensuremath{\psi}}				% psi
\newcommand{\tht}{\ensuremath{\theta}}				% theta
\newcommand{\what}{\ensuremath{\widehat{w}}}
\newcommand{\bB}{\ensuremath{\mathbf{B}}}	
\newcommand{\B}{\ensuremath{\mathbb B}}			%ball
\newcommand{\N}{\ensuremath{\mathbb N}}			%natural numbers
\newcommand{\R}{\ensuremath{\mathbb R}}			%reals
\newcommand{\C}{\ensuremath{\mathbb C}}			%complexes
\newcommand{\F}{\ensuremath{\mathbb F}}
\newcommand{\K}{\ensuremath{\mathbb K}}
\newcommand{\Z}{\ensuremath{\mathbb Z}}			%integers
\newcommand{\cZ}{\ensuremath{\mathcal Z}}
\newcommand{\om}{\ensuremath{\omega}}			% omega
\newcommand{\Om}{\ensuremath{\Omega}}			% cap omega
\newcommand{\gm}{\ensuremath{\gamma}}			% gamma
\newcommand{\Gm}{\ensuremath{\Gamma}}			% cap gamma
\newcommand{\D}{\ensuremath{\mathbb D}}			%disc
\newcommand{\T}{\ensuremath{\mathbb T}}			%torus
\newcommand{\cC}{\ensuremath{\mathcal C}}
\newcommand{\cD}{\ensuremath{\mathcal D}}
\newcommand{\cE}{\ensuremath{\mathcal E}}
\newcommand{\cF}{\ensuremath{\mathcal F}}
\newcommand{\cH}{\ensuremath{\mathcal H}}
\newcommand{\cL}{\ensuremath{\mathcal L}}
\newcommand{\cM}{\ensuremath{\mathcal M}}
\newcommand{\cN}{\ensuremath{\mathcal N}}
\newcommand{\cP}{\ensuremath{\mathcal P}}
\newcommand{\E}{\ensuremath{\mathbb E}}
\newcommand{\fG}{\ensuremath{\mathfrak G}}
\newcommand{\fN}{\ensuremath{\mathfrak N}}
\newcommand{\fp}{\ensuremath{\mathfrak p}}
\newcommand{\sig}{\ensuremath{\sigma}}				% sigma
\newcommand{\zt}{\ensuremath{\zeta}}				% zeta
\newcommand{\Dbar}{\ensuremath{\overline{\D}}}
\newcommand{\Ombar}{\ensuremath{\overline{\Om}}}
\newcommand{\norm}[1]{\lVert#1\rVert}				%norm
\newcommand{\twonorm}[1]{\lVert #1\rVert_2}			%norm_2
\newcommand{\pnorm}[1]{\lVert#1\rVert_p}			%norm_p
\newcommand{\onenorm}[1]{\lVert #1\rVert_1}			%norm_1
\newcommand{\supnorm}[1]{\lVert#1\rVert_{\infty}}		%norm_infinity
\newcommand{\ip}[2]{\langle#1, #2\rangle}			%inner product
\newcommand{\adj}{^*}							%Adjoint
\newcommand{\set}[1]{\left\{ #1\right\} }				%{}set
\newcommand{\seq}[1]{( #1)}						%{}sequence
\DeclareMathOperator{\Per}{Per}
\newcommand{\romannumbering}{\renewcommand{\theenumi}{\roman{enumi}}
\renewcommand{\labelenumi}{(\theenumi)}}			%enumerate using Roman nmbrs
\DeclareMathOperator{\clspan}{\overline{span}}
\newcommand{\emp}{\ensuremath{\varnothing}}		%empty set

\newcommand{\Ltwo}{\ensuremath{L^2}}				%L^2
\newcommand{\Lp}{\ensuremath{L^p}}				%L^p
\newcommand{\Linf}{\ensuremath{L^{\infty}}}			%L^infinity
\newcommand{\Lone}{\ensuremath{L^1}}				%L^1
\newcommand{\Htwo}{\ensuremath{H^2}}				%H^2
\newcommand{\Hone}{\ensuremath{H^1}}				%H^1
\newcommand{\Hp}{\ensuremath{H^p}}				%H^p
\newcommand{\Hinf}{\ensuremath{H^{\infty}}}			%H^infinity
\DeclareMathOperator{\Nev}{Nev}
\DeclareMathOperator{\LCM}{LCM}
\DeclareMathOperator{\GCD}{GCD}

\newcommand{\Lal}{L^{\al}}
\newcommand{\cLal}{\cL^{\al}}
\newcommand{\Hal}{H^{\al}}

\title{A Beurling Theorem for Generalized Hardy Spaces on a Multiply Connected Domain}
\author{Yanni Chen\qquad Don Hadwin\qquad Zhe Liu\qquad Eric Nordgren}
\date{June 30, 2016} 

\begin{document}
\maketitle

\begin{abstract}

The object of this paper is to prove a version of the Beurling-Helson-Lowdenslager invariant subspace theorem for operators on certain Banach spaces of functions on a multiply connected domain in $\C$. The norms for these spaces are either the usual Lebesgue and Hardy space norms or certain continuous gauge norms. 
In the Hardy space case the expected corollaries include the characterization of the cyclic vectors as the outer functions in this context, a demonstration that the set of analytic multiplication operators is maximal abelian and reflexive, and a determination of the closed operators that commute with all analytic multiplication operators.

\end{abstract}

%%%%%
\section{Introduction}

The setting for this investigation is a finitely connected domain $\Om$ in $\C$ with analytic boundary curves $\Gm$. The Lebesgue spaces are defined relative to harmonic measure $\om$ corresponding to an arbitrarily chosen point $\what$ in $\Om$. Versions of the Beurling \cite{B}, Helson-Lowdenslager \cite{HL} theorem in this context have appeared earlier in the work of Sarason \cite{Sarason}, Hasumi \cite{H}, Voichick  \cite{Voichick66}, and Rudol \cite{Rudol}, although by using Royden's  definition of inner function (see \cite{R}), we can write the theorem in the simpler more traditional form (see Theorems \ref{thm:BHL} and \ref{thm:BHLal}). Our version is modeled on the one obtained by Royden \cite{R} for Hardy spaces on a multiply connected domain. In addition to the added simplicity of the representation, it also allows us to address the matter of uniqueness.

In addition to the Lebesgue space $p$-norms, we also consider more general continuous gauge norms $\al$ on $\Linf(\Gm,\om)$ with an $\Lone(\Gm,\om)$ dominating property. This leads us to general Lebesgue spaces $\Lal(\Gm,\om)$ and Hardy spaces $\Hal(\Gm)$ where we obtain a general Beurling-Helson-Lowdenslager type invariant subspace theorem (see Theorem \ref{thm:BHLal}).

We begin in Section 2 by collecting some of the needed background. The domain $\Om$ has an analytic covering map $\tau$ from the unit disk $\D$ onto $\Om$ which induces a measure preserving transformation from the unit circle $\T$ onto $\Gm$, and consequently an isometric composition operator $C_{\tau}$ from the Lebesgue space $\Lp(\Gm,\om)$ for $1 \leq p \leq \infty$ into its counterpart $\Lp(\T,m)$ on the circle, where $m$ is normalized Lebesgue measure on $\T$. The Hardy spaces on $\Om$ were introduced by Parreau \cite{P} and Rudin \cite{Ru1} as consisting of analytic functions $f$ with $\abs{f}^p$ dominated by some harmonic function. As on the disk, these functions have boundary limits, and hence the spaces $\Hp(\Om)$ may be identified with isometrically isomorphic subspaces $\Hp(\Gm)$ of $\Lp(\Gm,\om)$. The Hardy space theory of the unit circle has been extended by the first author \cite{Chen1, Chen2} by considering norms that are more general than the Lebesgue norms and these are introduced in section \ref{subsec:gauge}.

Although the Hardy space theory of a multiply connected domain is in large part similar to that of the unit disk (see Royden \cite{R} sections 1 and 2), the multiple connectivity introduces some interesting differences. In particular, not all harmonic functions on $\Om$ have single-valued harmonic conjugates because the holes in $\Om$ may give rise to periods and multiple valued harmonic conjugates. A consequence, for example, is that whereas on the unit circle $\Ltwo(\T,m)$ has an orthogonal direct sum decomposition into a subspace of analytic functions, a subspace of constant functions, and a subspace of co-analytic functions, on $\Gm$ the space $\Ltwo(\Gm,\om)$ has a similar decomposition, but with an additional $n$-dimensional subspace resulting from the possibility of periods from each of the $n$ holes. Also, inner functions in this context turn out to be multiple valued if the restriction that their boundary values need to be unimodular is enforced, but if we follow Royden \cite{R} and relax the boundary condition to that of constant modulus on each of the connected components of $\Gm$, then the multiple value problem can be made to go away.

The $\Ltwo$ case of the Beurling-Helson-Lowdenslager theorem is dealt with in Section 3, where we make use of the Forelli~\cite{Forelli} projection operator $\cP$ which maps $\Lone(\T,m)$ onto the range of $C_{\tau}$ acting on $\Lone(\Gm,\om)$. It is shown that if $\cM$ is a subspace of $\Ltwo(\Gm,\om)$ that is invariant under multiplication by every function in $\Hinf(\Gm)$, then the image of $\cM$ under $C_{\tau}$ can also be obtained by applying $\cP$ to the invariant subspace of the unilateral shift generated by $C_{\tau}(\cM)$. This fact is used to show that these subspaces $\cM$ are either of the form $\chi_{\E}\Ltwo(\Gm,\om)$ or $\ph\Htwo(\Gm)$, where $\E$ is a measurable subset of $\Gm$ and $\ph$ is a function on $\Gm$ having constant modulus on each component of $\Gm$.  (In a private communication Alexandru Aleman \cite{A} has indicated that he has also obtained this result by different means on a region where the boundary curves are circles.) As a corollary we obtain a special case of Royden's result \cite[Theorem 1]{R} that when $\cM$ is included in $\Htwo(\Gm)$, then $\ph$ is inner. It is also shown that in this context as well as that of the circle the cyclic vectors of the set of multiplications by $\Hinf(\Gm)$ functions on $\Htwo(\Gm)$ are the outer functions.

The $\Lal(\Gm,\om)$ case is dealt with in Section 4 by using a slight modification of the proof of the first author in \cite{Chen2} (see also \cite{G1}). Every continuous, dominating, gauge norm on $\Linf(\Gm,\om)$ induces the same topology as the weak* topology on the ball of $\Linf(\Gm,\om)$, and this coincides with the topology of convergence in measure on the ball. As in the case of the unit disk, the space $\Hal(\Gm)$ consists of the members of $\Hone(\Gm)$ that are also in $\Lal(\Gm,\om)$, and thus members of $\Linf(\Gm,\om)$ having reciprocals in $\Lal(\Gm,\om)$ differ from outer functions by functions of locally constant modulus. This makes it possible to use the $\Ltwo(\Gm,\om)$ result to show that the invariant subspaces of the $\Hinf(\Gm)$ multiplication operators on $\Lal(\Gm,\om)$ have the same form as those in $\Htwo(\Gm)$. Consequently, invariant subspaces in $\Hal(\Gm)$ also  are determined by inner functions $\ph$ and have the form $\ph\cdot\Hal(\Gm)$, a result that contains the full version of Royden's invariant subspace theorem mentioned in the preceding paragraph. 

Section 5 establishes some properties of inner functions on the basis of invariant subspaces. On the disk the only simultaneously inner and outer functions are constants, but on a multiply connected domain being both inner and outer is equivalent to being inner and invertible, and the group of such functions is quite substantial. Also in this section the cyclic vectors in $\Hal(\Gm)$ are characterized as outer functions.

Section 6 shows that the spaces $\Hal(\Gm)$ fit into the multiplier pair context of \cite{HN} and consequently the algebra of multiplication operators by $\Hinf(\Gm)$ functions is maximal abelian and reflexive. The paper concludes in Section 7 with a characterization of the closed operators on $\Hal(\Gm)$ that commute with the analytic multiplication operators.

%\newpage
%%%%%
\section{Preliminaries} \label{sec:prelim}

In this section we review some known facts for later use.

\subsection{Hardy spaces on $\Om$}
Let $\Om$ be a bounded multiply connected domain in $\C$ with analytic boundary curves $\Gm_{0}, \ldots, \Gm_{n}$. Assume that $\Om \subset \interior \Gm_{0}$, and let $\Om_{j}$ be $\interior \Gm_{0}$ when $j = 0$ and $\exterior \Gm_{j}$ in the extended complex plane when $j > 0$. Also assume that for $1 \leq k \leq n$ the sets $\Gm_{k}$ together with their interiors are pairwise disjoint subsets of $\Om_{0}$. Thus $\Om = \bigcap_{j=0}^{n} \Om_{j}$. Fix a point $\what$ in $\Om$ and let $\Phi_{j}$ be the Riemann mapping function that sends the open unit disk $\D$ onto $\Om_{j}$ with $\Ph_{j}(0) = \what$ for $0 \leq j \leq n$. Because the boundary of $\Om$ consists of analytic curves, the functions $\Ph_{j}$ are analytic on $\Dbar$. We will treat the $\Gm_{j}$ as parameterized curves when convenient with parametrizations given by $\Gm_{j}(t) = \Ph_{j}(e^{it})$ for $0 \leq t \leq 2\pi$. It follows that with $\Gm = \Gm_{0} + \cdots + \Gm_{n}$ the points of $\Om$ have index one relative to the cycle $\Gm$ and points of the complement of $\Ombar$ have index 0.

The parameterizations $\Gm_{j}$ give rise to arc length measure on $\Gm$ defined by $\ds = \abs{\Gm_{j}'(t)}\dt$ at points $\Gm_{j}(t)$ of $\Gm$. (More precisely, arc length measure on $\Gm$ is the measure that is obtained  on each component $\Gm_j$ of $\Gm$ separately by lifting the measure on $[0,2\pi]$ that has Radon-Nikodym derivative $\abs{\Gm_j'(t)}$ relative to Lebesgue measure using the map $t \mapsto \Ph_j(e^{it})$.)  But there exists a measure $\om$ on $\Gm$ that is better adapted to our needs and can be related to normalized Lebesgue measure $m$ on $\T$ by means of the Koebe mapping function $\tau:\D \ra \Om$. The function $\tau$ is analytic, surjective, locally one to one, and it has the property that every point $w \in \Om$ lies in a disk $D_{w}$ whose inverse image under $\tau$ is made up of connected components that are each mapped bijectively by $\tau$ onto $D_{w}$ (see Conway \cite[Chapter 16]{C}). Also, $\tau$ can be chosen so that $\tau(0) = \what$, and the additional requirement $\tau'(0) > 0$ makes $\tau$ unique. Since $\Gm$ is made up of analytic curves, there exists an open subset $\T_{0}$ of the unit circle $\T$ over which $\tau$ has an analytic continuation, such that $\tau(\T_{0}) = \Gm$, and such that the complement of $\T_{0}$ in $\T$ has Lebesgue measure 0 (see Tsuji \cite[Theorem XI. 17]{T}).

Every continuous function $f$ on $\Gm$ has a continuous extension to $\Ombar$ that is harmonic on $\Om$, and which we also label $f$. By the maximum principle, evaluation at a point $w$ of $\Om$ is a continuous linear functional on the space $C(\Gm)$ of continuous complex functions on $\Gm$. The Riesz representation theorem implies the existence of a probability measure $\om_{w}$ on the Borel subsets of $\Gm$ such that $f(w) = \int_{\Gm} f \domega_{w}$. It can be shown that 
\begin{equation} \label{eqn:omandGreenfn}
\frac{\domega_{w}}{\ds} = \frac{1}{2\pi} \pardrv{g_{w}}{n},
\end{equation}
where $g_{w}$ is the Green's function for $\Om$ with pole at $w$ and the derivative is in the direction of the inward pointing normal. More importantly for our purposes, the measure $\om_{w}$ can also be related to normalized Lebesgue measure $m$ on $\T$ as follows. If $f$ is a continuous function on $\Ombar$ that is harmonic on $\Om$, then $f(w) = \int_{\Gm} f \domega_{w}$, and $f\of\tau$ is a bounded harmonic function on $\D$ with boundary values also given $m$-a.e.\ by $f\of\tau$. For $z \in \D$, let $m_{z}$ be defined by $\dm_{z} = P_{z} \dm$, where $P_{z}$ is the Poisson kernel for evaluation at $z$. Thus if $\tau(z) = w$, then $\intT f\of\tau \dm_{z} = f(\tau(z)) = f(w)$, and consequently the transformation of integral formula implies $\om_{w} = m_{z}\tau\inv$. We state this as a lemma for easy reference.

%%%
\begin{lemns} \label{lem:hameas}
If $z \in \D$ and $w = \tau(z)$, then $\om_{w} = m_{z}\tau\inv$.
\end{lemns}

Knowledge of the measures $\om_{w}$ makes the solution of the Dirichlet problem explicit; if $f \in C(\Gm)$, then for $w \in \Om$, $f(w) = \int_{\Gm} f \domega_{w}$ gives the harmonic function on $\Om$ having the original $f$ as its boundary function. The relation between $m_{z}$ and $\om_{w}$ makes it easy to derive properties of $\om_{w}$ from corresponding ones for $m_{z}$. The following is an example. 

%%%
\begin{corns}
Each of the measures $\om_{w}$ is boundedly mutually absolutely continuous with respect to arc length measure.
\end{corns}

The case of $z = 0$, and hence $m_{z} = m$, is particularly important. In this case we will write $\what$ for $\tau(0)$ and simply $\om$ for $\om_{\what}$. Thus it follows from the preceding that the operator $C_{\tau}$ of composition with $\tau$ maps each of the spaces $\Lp(\Gm,\om)$ ($1 \leq p \leq \infty$) isometrically into $\Lp(\T,m)$, and there is an expectation operator $\cE$ on $\Lp(\T,m)$ having the same range as $C_{\tau}$. When $p = 2$ the operator $\cE$ is the orthogonal projection of $\Ltwo(\T,m)$ onto the range of $C_{\tau}$. If it is only assumed that $f \in \Lone(\Gm,\om)$, then the function on $\Om$ given by $f(w) = \int_{\Gm} f \domega_{w}$ is harmonic and can be shown to have nontangential boundary limits given by $f$ $\om$-a.e. 

%\gr{$C(\Gm)$}
%\gr{$m$, $\T$, $m\tau\inv$, $C_{\tau}$, $\Lp(Gm,\om)$}

Let $\Hinf(\Om)$ be the space of bounded analytic functions on $\Om$. The mapping $C_{\tau}$ can also be thought of as a transformation of functions on $\Om$ into functions on $\D$, and as such transforms $\Hinf(\Om)$ into a subspace of $\Hinf(\D)$. The set $\fG$ of all disk automorphisms $\sig$, i.e.\ linear fractional transformations of $\D$ onto itself, that satisfy $\tau\of\sig = \tau$ is the covering group of $\tau$. It has the following useful property. 

%%%
\begin{lemns} \label{lem:Ginv}
A measurable function $F$ on $\T$ has the form $F = f \of \tau$ for a measurable function $f$ on $\Gm$ if and only if $F \of \sig = F$ for all $\sig \in \fG$.
\end{lemns}

Since functions in $\Hinf(\D)$ have non tangential limits a.e.\ on $\T$ and $\tau$ is analytic on $\T_{0}$, it follows that all functions in $\Hinf(\Om)$ have non tangential limits $\om$-a.e.\ on $\Gm$. Furthermore, it also follows that the Nevanlinna class $\Nev(\Om)$, consisting of all analytic functions on $\Om$ that are quotients of functions in $\Hinf(\Om)$, has the property that all its members also have non tangential limits $\om$-a.e.\ on $\Gm$. These limits define their boundary functions which constitute the class $\Nev(\Gm)$, and $\Hinf(\Gm)$ is the subspace of boundary functions of members of $\Hinf(\Om)$. Thus $\Nev(\Gm)$ shares with $\Nev(\T)$ the property that the vanishing of one of its members on a set of positive measure entails its vanishing almost everywhere.

The space $\Hp(\Om)$ for $0 < p < \infty$ is defined to consist of all analytic functions $f$ on  $\Om$ such that $\abs{f}^{p} \leq h$ for some harmonic function $h$ (see \cite{Ru1}). In this case there is a smallest such harmonic function $h$, the least harmonic majorant of $\abs{f}^{p}$, and it is used to define $\pnorm{f}$ as $h(\what)^{1/p}$, which is a norm making $\Hp(\Om)$ into a Banach space when $1 \leq p <\infty$.

Suppose $f \in \Hp(\Om)$ and $h$ is a harmonic function satisfying $\abs{f}^{p} \leq h$ on $\Om$. If $f_{1} = f\of\tau$ and $h_{1} = h\of\tau$, then $f_{1}$ is analytic on $\D$, $h_{1}$ is harmonic on $\D$, and $\abs{f_{1}}^{p} \leq h_{1}$. It follows that $f_{1} \in \Hp(\D)$. (Reason: if $(f_{1})_{r}(z) = f_{1}(rz)$, then $\intT \abs{(f_{1})_{r}}^{p} \dm \leq \intT h_{1} \dm = h_{1}(0)$ for all $r \in [0,1)$). Taking the supremum over $r$ shows that $\pnorm{f_{1}} \leq h_{1}(0)^{1/p}$.  If $h$ is the least harmonic majorant of $\abs{f}^{p}$, then, as Rudin~\cite{Ru1} showed, $h_{1}$ is the least harmonic majorant of $\abs{f_{1}}^{p}$. Here is the argument: call the least harmonic majorant of $\abs{f_{1}}^{p}$ for the moment $h_{0}$, so $h_{0} \leq h_{1}$. If $\sig \in \fG$, then $\abs{f_{1}\of\sig}^{p} \leq h_{0}\of\sig$, and hence $\abs{f_{1}}^{p} \leq h_{0}\of\sig$ implying $h_{0} \leq h_{0}\of\sig $. Because $\fG$ is a group, it follows that $h_{0}\of\sig = h_{0}$ for all $\sig \in \fG$, and thus $h_{0} = h_{2}\of\tau$ for some  harmonic function $h_{2}$ that majorises $\abs{f}^{p}$ on $\Om$. Consequently $h \leq h_{2}$, it follows that $h_{1} \leq h_{0}$, and the argument is complete. Since $h_{1}$ is the least harmonic majorant of $f_{1}$, $\pnorm{f_{1}} = \pnorm{f}$, and therefore $C_{\tau}$ maps $\Hp(\Om)$ isometrically into $\Hp(\D)$. Moreover the image of $\Hp(\Om)$ under $C_{\tau}$ consists of all functions $f_{1}$ in $\Hp(\D)$ satisfying $f_{1}\of\sig = f_{1}$ for all $\sig \in \fG$. This observation could have been used to establish that $\Hp(\Om)$ is a Banach space when $1 \leq p < \infty$ on the basis of the known fact that $\Hp(\D)$ is a Banach space.

Just as each of the Hardy spaces $\Hp$ on the disk is isometrically isomorphic to its space of boundary functions, the same is true for the region $\Om$. Let $f $ be in $\Hp(\Om)$ and, as above, put $f_{1} = f\of\tau$. Then $f_{1}$ is in $\Hp(\D)$ and it has a boundary function $\hat{f}_{1}$ defined almost everywhere on $\T$. Since $f_{1}\of\sig = f_{1}$ for all $\sig \in \fG$ it follows that $\hat{f}_{1}\of\sig = \hat{f}_{1}$ on $\T$, and therefore $\hat{f}_{1} = \hat{f}\of\tau$ for some function $\hat{f}$ on $\Gm$ which belongs to $\Lp(\Gm,\om)$ and has the same norm as $\hat{f}_{1}$. If $\zt \in \T_{0}$, $w_{b} = \tau(\zt)$, and $f_{1}$ has nontangential limit $\hat{f}_{1}(\zt)$ at $\zt$, then, because $\tau$ is analytic at $\zt$, it follows that $f$ has nontangential limit $\hat{f}_{1}(\zt) = \hat{f}(\tau(\zt)) = \hat{f}(w_{b})$, and thus $f$ has boundary function $\hat{f}$ at almost every point $w_{b}$ of $\Gm$. The process of taking limits of an $\Hp(\Om)$ function at boundary points is reversed by forming integrals $w \mapsto \int_{\Gm} f \domega_{w}$, which are analogous to the Poisson integrals on the disk. In summary, the space $\Hp(\Om)$ can be viewed equivalently as the collection of analytic functions $f$ on $\Om$ for which $\abs{f}^{p}$ has a harmonic majorant, or as the isometrically isomorphic subspace $\Hp(\Gm)$ of $\Lp(\Gm,\om)$ consisting of the boundary functions of members of $\Hp(\Om)$, or as the subspace of $\Hp(\D)$ consisting of those functions invariant under composition with all members of $\fG$, or as the subspace of $\Hp(\T)$ with the same invariance property relative to $\fG$. We will make use of these different views more or less interchangeably. Moreover it can be shown that the rational functions with poles off $\Ombar$ are dense in $\Hp(\Om)$ (weak* when $p = \infty$.)

Although largely similar to the Hardy space theory of the unit disk, the theory for a multiply connected domain $\Om$ differs significantly in one respect. We will confine ourselves here mainly to discussing the $\Ltwo$ case, which we will need in Section~\ref{sec:LtwoBHL}. In the case of the disk there is the familiar decomposition $\Ltwo(\T,m) = \Htwo(\T) \oplus \Htwo_{0}(\T)\adj$, where $\Htwo_{0}(\T)\adj$ is the set of complex conjugates of the functions in $\Htwo(\T)$ that are orthogonal to 1 (the set of functions vanishing at 0). The counterpart of this decomposition for $\Om$ is $\Ltwo(\Gm,\om) = \Htwo(\Gm) \oplus \Htwo_{0}(\Gm)\adj \oplus N(\Gm)$, where $N(\Gm)$ is an $n$-dimensional subspace of bounded functions, and $\Htwo_{0}(\Gm)\adj$ is the set of complex conjugates of of the functions in $\Htwo(\Gm)$that are orthogonal to 1 (the set of functions vanishing at $\what$). The subspace $N(\Gm)$ is the span of functions $Q_{1}, \ldots , Q_{n}$, and it is orthogonal to the set of real parts of the rational functions with poles outside of $\Ombar$. The exact specification of the $Q_{j}$ is given by $Q_{j} \domega = \pardrv{h_{j}}{n} \ds$, where each $h_{j}$ is the harmonic function on $\Om$ with boundary values 1 on $\Gm_{j}$ and 0 on $\Gm_{k}$ with $k \ne j$. These functions will be encountered again in Section~\ref{subsec:eigen}, but for a full discussion see Fisher's book \cite{Fisher}, sections 4.2 and 4.5. The subspace $N(\Gm)$ that they span is of importance because $C_{\tau}$ maps not only $\Htwo(\Gm)$ into $\Htwo(\T)$, but in fact $C_{\tau}\bigl(\Htwo(\Gm) \oplus N(\Gm) \bigr) = \cE\bigl(\Htwo(\T)\bigr)$ \cite{Forelli}.

%%%%%
\subsection{Gauge norms}\label{subsec:gauge}

In \cite{Chen1} the first author introduced the study of Hardy spaces on $\T$ under a family of norms that properly includes the $p$\,-norms. Since our interest is in the space $\Gm$ with the measure $\om$, we will introduce norms of this type in a more general setting. Let $\mu$ be a nonatomic probability measure on a $\sig$-algebra in a set $X$, and let $\al$ be a norm on $\Linf(X,\mu)$. We call $\al$ a \emph{gauge norm} in case $\al(1) = 1$ and $\al(\abs{f}) = \al(f)$ for all $f \in \Linf(X,\mu)$, and we say it is \emph{continuous} in case 
\begin{equation*}
\lim_{\mu(\E)\ra 0} \al(\chi_{\E}) = 0.
\end{equation*}
Also, we will call $\al$ \emph{dominating} in case $\onenorm{f} \leq \al(f)$ whenever $f \in \Linf(X,\mu)$. The property should more properly be called \emph{one-norm} dominating, but we will use the shorter locution. It was shown in \cite{Chen1}, Proposition 2.2, that if a continuous gauge norm on $\Linf(\T,m)$ is rotationally symmetric in the sense that $\al(f_{\tht}) = \al(f)$ for all $\tht$ where $f_{\tht}(z) = f(e^{i\tht}z)$ for all $f \in \Linf(\T,m)$, then $\al$ is dominating.

A gauge norm $\al$ may be extended to all measurable complex functions $f$ on $X$ by 
\begin{equation*}
\al(f) = \sup\set{\al(s): s \text{ is a simple function and } 0 \leq s \leq \abs{f}}.
\end{equation*}
Let $\cLal(X,\mu)$ consist of all measurable functions $f$ such that $\al(f) < \infty$. If $\al$ is a continuous dominating gauge norm on $\Linf(X,\mu)$, then its extension to $\cLal(X,\mu)$ has the same properties. The space $\cLal(X,\mu)$ is a Banach space, and we define $\Lal(X,\mu)$ to be the closure of $\Linf(X,\mu)$ in $\cLal(X,\mu)$.

Let $\al$ be a dominating, gauge norm, and define its \emph{dual norm} $\al'$ on $\Linf(X,\mu)$ by 
\begin{equation*}
\al'(f) = \sup\set{\Bigl| \int_{X} fh \dmu \Bigr|: h \in \Linf(X,\mu) \text{ and }\al(h) \leq 1}.
\end{equation*}

The following are Lemma 2.6 and Proposition 2.7 of \cite{Chen2}.

\begin{lemns}
The dual norm $\al'$ of a dominating gauge norm $\al$ is also a dominating gauge norm.
\end{lemns}

\begin{propns} \label{prop:dualLal}
Suppose $\al'$ is the dual norm of a dominating gauge norm $\al$ on $\Linf(X,\mu)$. The dual space $\bigl(\Lal(X,\mu)\bigr)^{\#}$ is $\cL^{\al'}(X,\mu)$ in the sense that if $\ph$ is a continuous linear functional on $\Lal(X,\mu)$, then there exists a unique $F \in \cL^{\al'}(X,\mu)$ satisfying $\norm{\ph} = \al'(F)$ such that for all $f \in \Lal(X,\mu)$, $fF \in \Lone(X,\mu)$ and
\begin{equation*}
\ph(f) = \int_{X} fF \dmu.
\end{equation*}
\end{propns}

Throughout the rest of the paper, without explicit assumption to the contrary, $\al$ will be assumed to be a continuous, dominating, normalized gauge norm on $\Lal(\Gm,\om)$. The set of these norms constitute a set that we will label $\fN$. Also $\fN_{\infty}$ will be $\fN$ with the essential supremum norm adjoined.

Hardy spaces in this context are obtained by defining $\Hal(\Gm)$ to be the subspace of $\Lal(\Gm,\om)$ obtained by taking the $\al$-norm closure of $\Hinf(\Gm)$. Since $\Lal(\Gm,\om)$ is a closed subspace of $\Lone(\Gm,\om)$, $\Hal(\Gm)$ is a closed subspace of $\Hone(\Gm)$. Thus we may define $\Hal(\Om)$ as the subspace of $\Hone(\Om)$ consisting of those functions whose boundary functions are in $\Hal(\Gm)$. For $f \in \Hal(\Om)$ and $w \in \Om$ we have $f(w) = \int_{\Gm} f \domega_{w}$, and since $\om_{w}$ is boundedly absolutely continuous with respect to $\om$, it follows from the dominating property that point evaluations are continuous linear functionals on $\Hal(\Om)$ and by extension on $\Hal(\Gm)$. Thus $\Hal(\Om)$ is a functional Banach space, and provides an equivalent but different view to $\Hal(\T)$.

%%%%%
\subsection{Harmonic functions, periods, and harmonic conjugates}\label{subsec:eigen}

For each $j$ between 0 and $n$, let $h_{j}$ be the solution to the Dirichlet problem on $\Om$ corresponding to the boundary function $\chi_{\Gm_{j}}$, and note that, because the boundary curves are analytic, each $h_{j}$ has a harmonic extension to an open neighborhood of $\Ombar$. Hence for $0\leq j \leq n$ we have $h_{j}(w) = \int_{\Gm} \chi_{\Gm_{j}} \domega_{w} = \om_{w}(\Gm_{j})$. We will call a real linear combination of the functions $h_{j}$ with $1 \leq j \leq n$ a \emph{harmonic unit}. (Royden~\cite{R} calls these functions harmonic \emph{measures}, but we will reserve that term for the actual measures $\om_{w}$ introduced above.) Thus if $\vec{a} = (a_{1}, \ldots, a_{n}) \in \R^{n}$, then the typical harmonic unit is the function $u_{\vec{a}} = \sum_{j=1}^{n} a_{j}h_{j}$. Observe that a harmonic unit plus a constant gives the most general linear combination of all the $h_{j}$ for $0 \leq j \leq n$. 

On the disk every harmonic function has a harmonic conjugate, but on an annulus centered at 0 for example, the harmonic function $u(w) = \log\abs{w}$ does not have a single-valued harmonic conjugate, and thus there is no analytic function on the entire annulus that has $u$ as its real part. However, the following lemma shows that one can construct an analytic function $f$ on $\Om$ from any given harmonic function $u$ by modifying $u$ by the addition of an appropriately chosen harmonic unit.

If $u$ is harmonic and real-valued on $\Om$, then let $v$ be the harmonic conjugate of $u$ on $D_{\what}$ satisfying $v(\what) = 0$. Then $f = u + iv $ is an analytic function on $D_{\what}$. If $\gm$ is a loop in $\Om$ at $\what$, i.e.\ a path in $\Om$ with $\what$ as both initial and terminal points, then $f$ can be continued analytically along $\gm$ to produce a second holomorphic function $f_{\gm} = u + iv_{\gm}$ on $D_{\what}$. The difference $f - f_{\gm} = i(v - v_{\gm})$ is both holomorphic and pure imaginary and therefore constant with the value $iv_{\gm}(\what)$. The real number $\Per(u,\gm) = v_{\gm}(\what)$ is the period of the harmonic conjugate of $u$ on $\gm$. It is not hard to see that $u$ will have a (single valued) harmonic conjugate on $\Om$ precisely when analytic continuation of $f$ along a path in $\Om$ depends only on the end points of the path, and this condition is equivalent to $\Per(u,\gm) = 0$ for every loop $\gm$ at $\what$.

The period $\Per(u,\gm)$ can be expressed in terms of $u$ with the aid of the Cauchy-Riemann equations 
\begin{equation} \label{eq:Greenu}
\Per(u,\gm) = \int_{\gm} \pardrv{v}{s} \ds = - \int_{\gm} \pardrv{u}{n} \ds,
\end{equation}
where $\pardrv{\ }{s}$ indicates the directional derivative in the direction of the unit tangent vector $\vec{t} = \gm'(t)/\abs{\gm'(t)}$ of $\gm$ at $\gm(t)$, and $\pardrv{\ \,}{n}$ indicates the directional derivative in the direction of the interior unit normal vector $\vec{n} = i\vec{t}$. Using Green's formulas, one can see that the second integral in \eqref{eq:Greenu} is constant as $\gm$ varies over its homotopy equivalence class in $\Om$, and thus the period is a function defined on the fundamental group $\pi_{1}(\Om)$ of $\Om$, which is a free group on $n$ generators, and $\Per$ is a homomorphism of $\pi_{1}(\Om)$ into $\R$. Hence to determine if $u$ has a harmonic conjugate on $\Om$ it suffices to check that $\Per(u,\gm) = 0$ for each $\gm$ in a set of generators for $\pi_1(\Om)$, and these may be taken to be curves $\gm_{j}$ for $1 \leq j \leq n$, each with $\Gm_{j}$ in its interior, with $\Gm_{k}$ in its exterior when $k \ne j$, and such that each point interior to $\Gm_{j}$ has winding number one.

%%%%
\begin{lemns} \label{lem:Per}
If $u$ is a real-valued harmonic function on $\Om$, then there exists a harmonic unit $u_{\vec{a}}$ such that $u + u_{\vec{a}}$ is the real part of an analytic function on $\Om$. 
\end{lemns}
 
 %%%
\begin{proof}
If $u$ is harmonic on $\Om$ and $u_{\vec{a}}$ is a harmonic unit, then, by the above discussion, it will suffice to show that $\vec{a}$ can be chosen so that $\Per(u + u_{\vec{a}},\gm_j) = 0$ for $0 \leq j \leq n$, which by equation \eqref{eq:Greenu} translates into 
 $\int_{\gm_{j}} \pardrv{u_{\vec{a}}}{n} \ds = -\int_{\gm_{j}} \pardrv{u}{n} \ds$ for $1 \leq j \leq n$. Consider the integral 
$\int_{\gm_{j}} \pardrv{u_{\vec{a}}}{n} \ds = \sum_{k=1}^{n} a_{k} \int_{\gm_{j}} \pardrv{h_{k}}{n} \ds$. Each $\gm_{j}$ is homotopic to $\Gm_{j}$ in $\Ombar$, and since $h_{k}$ is harmonic on $\Ombar$, 
$$\int_{\gm_{j}} \pardrv{h_{k}}{n} \ds = \int_{\Gm_{j}} \pardrv{h_{k}}{n} \ds= \int_{\Gm} h_{j} \pardrv{h_{k}}{n} \ds.$$ 
The $n\times n$ period matrix with entries $p_{j,k} =  \int_{\Gm} h_{j} \pardrv{h_{k}}{n} \ds$ is known to be symmetric and invertible (see \cite[page 80]{Fisher} and \cite[pages 38 -- 41]{N}). Thus the condition that $u$ imposes on $\vec{a}$ is that the system of equations 
$$ \sum_{k=1}^{n} p_{j,k} a _{k} = -\int_{\gm_{j}} \pardrv{u}{n} \ds \qquad\qquad 1 \leq j \leq n$$
has a solution, which it does by the invertibility of the period matrix.
\end{proof}

%%%%%
\subsection{Outer functions and Eigenfunctions of the group $\fG$}

In \cite{R} Royden defines an \emph{inner function} on $\Om$ as a bounded analytic function having a boundary limit function on $\Gm$ that has $\om$-a.e.\ constant modulus on each boundary component $\Gm_{j}$ for $0 \leq j \leq n$. He also calls an analytic function $f \in \Nev(\Om)$ \emph{outer} in case $\log\abs{f(w)} = \int_{\Gm} \log\abs{f} \domega_{w}$. If $w = \tau(z)$, then $\om_{w} = m_{z}\tau\inv$ by Lemma~\ref{lem:hameas}, and thus his condition becomes
\begin{equation*}
\log\abs{f\of\tau(z)} = \int_{\Gm} \log\abs{f} \dm_{z}\tau\inv = \intT \log\abs{f\of\tau} \dm_{z},
\end{equation*}
which is the condition that $f\of\tau$ be outer on $\D$. Thus a function $f$ is outer on $\Om$ in Royden's sense if and only if $f\of\tau$ is outer on $\D$ in the usual sense, and it follows, as in the case of the unit disk, that Royden's condition holds for all $w \in \Om$ if and only if it holds for a single point. This establishes the following lemma.

%%%
\begin{lemns}\label{lem:outer}
A function $f \in \Hone(\Gm)$ is outer if and only if $f\of\tau$ is outer in $\Hone(\T)$.
\end{lemns}

We remark that the above argument showing that $f$ is outer in $\Hone(\Gm)$ if and only  if $f\of\tau$ is outer in $\Hone(\T)$ also shows that Jensen's inequality for an $\Hone(\T)$ function implies the corresponding inequality for functions in $\Hone(\Gm)$. For if $f \in \Hone(\Gm)$, then $f\of\tau \in \Hone(\T)$, and consequently for $z \in \D$, $\log\abs{f\of\tau(z)} \leq \intT \log\abs{f\of\tau} \dm_{z}$ implying $\log\abs{f(w)} \leq \int_{\Gm} \log\abs{f} \domega_{w}$. 

Additionally, we note that if $u$ is an integrable real-valued function on $\Gm$ and the harmonic function defined for $w \in \Om$ by $u(w) = \int_{\Gm} u \domega_{w}$ has a harmonic conjugate $v$ on $\Om$, then the function $f = \exp(u + iv)$ on $\Om$ is outer.

A basic tool is the following lemma which will be needed on several occasions. It is essentially Forelli's Lemma 5 in \cite{Forelli} with a slightly more explicit description of the eigenfunctions.

%%%
\begin{lemns}\label{lem:eigenfneta}
If $\eta$ is a character of $\fG$ (i.e.\ a homomorphism of $\fG$ into $\T$), then there exists an invertible outer function $F$ in $\Hinf(\T)$ such that $\abs{F}$ is constant $m$-a.e.\ on each of the sets $\tau\inv(\Gm_{k})$ for $0 \leq k \leq n$ and for every $\sig \in \fG$, $F\of\sig = \eta(\sig) F$.
\end{lemns}

%%%
\begin{proof}  
To construct the required function $F$ we begin with a harmonic unit $u = \sum_{k=1}^{n} a_{k}h_{k}$ where each $a_{k}$ is a real constant that remains to be specified. Let $U = u\of\tau$ to obtain a bounded harmonic function on $\D$ with boundary function $U = \sum_{k=1}^{n} a_{k} \chi_{\tau\inv(\Gm_{k})}$. 

If $\cC$ is the harmonic conjugation operator on $\Ltwo(\T,m)$, then the matrix of $\cC$ relative to the usual orthonormal basis for $\Ltwo(\T,m)$ is diagonal with negatively indexed entries $i$, positively indexed entries $-i$, and 0 as entry at 0, so we observe that $\cC\adj = -\cC$. Put $V = \cC(U)$ and $F = \exp(U + iV)$. Then $F$ is outer in $\Hinf(\T)$ and $\abs{F} = \sum_{k=0}^{n} e^{a_{k}} \chi_{\tau\inv(\Gm_{k})}$, where $a_{0}=0$, so it will fulfill the requirements provided the condition $F\of\sig = \eta(\sig) F$ for all $\sig \in \fG$ is satisfied. The task at hand is to show that for $1 \leq k \leq n$ the $a_{k}$ can be chosen to satisfy this condition.

Observe that however the $a_{k}$'s are chosen, we have that $\sig \in \fG$ implies that $(F\of\sig)/F = \exp{i(V\of\sig - V)}$ because $U\of\sig = U$. The equation has an analytic function on the left side and a function taking values in $\T$ on the right, and thus these functions are constant, with the constant, say $\eta_{1}(\sig)$, dependent upon $\sig$. Clearly $\eta_{1}(\sig) = 1$ when $\sig(z) = z$, the identity of the group $\fG$, and because $\eta_{1}(\sig) = (F\of\sig)/F$, it follows that $\eta_{1}$ is a homomorphism of $\fG$ into $\T$. Thus it remains to show that the $a_{k}$ can be chosen so that $\eta_{1} = \eta$.

The group $\fG$ is isomorphic to the fundamental group $\pi_{1}(\Om)$. If $\set{\sig_{1}, \sig_{2},\ldots,\sig_{n}}$ is a set of generators of $\fG$, then since the values $\eta_{1}(\sig_{j})$ for $1 \leq j \leq n$ completely determine $\eta_{1}$, it suffices to show that by an appropriate choice of the $a_{k}$, we have $\eta_{1}(\sig_{j}) = \eta(\sig_{j})$ for $1 \leq j \leq n$. Suppose $\eta(\sig_{j}) = e^{i\tht_{j}}$. Since $\eta_{1}(\sig_{j}) = \exp{i(V\of\sig - V)} = e^{iV(\sig_{j}(0))}$, the requirement is that the $a_{k}$ can be chosen so that the system of equations $V(\sig_{j}(0)) = \tht_{j}$ for $1\leq j \leq n$ has a solution.

Observe that if $P_{j}$ is the Poisson kernel for evaluation at $\sig_{j}(0)$, then
\begin{equation}\label{eqn:vsig}
V(\sig_{j}(0)) = \ip{V}{P_{j}} = \ip{U}{\cE\bigl(\cC(-P_{j})\bigr)},
\end{equation}
where $\ip{\cdot}{\cdot}$ is the inner product on $\Ltwo(\T,m)$, and, as before, $\cE$ is the projection on the range of $C_{\tau}$. In Fisher~\cite{Fisher} it is shown that if $\gm_{j}$ is the loop chosen in the proof of Lemma~\ref{lem:Per}, then $\gm_{j}$ can be lifted to a curve in $\D$ with initial point 0 and terminal point, say $\zt_{j}$, and if $\sig_{j}$ is the unique member of $\fG$ satisfying $\sig_{j}(0) = \zt_{j}$, then the $\sig_{j}$ so chosen form a set of generators for $\fG$, and $\cE\cC(-P_{j}) = C_{\tau}(Q_{j})$, where $Q_{j}\domega = \pardrv{h_{j}}{n}\ds$. Thus equation \eqref{eqn:vsig} leads to 
\begin{equation*}
V(\sig_{j}(0)) = \ip{C_{\tau}(u)}{C_{\tau}(Q_{j})} = \sum_{k=1}^{n}a_{k}\ip{h_{k}}{Q_{j}}.
\end{equation*}
Note that $\ip{h_{k}}{Q_{j}} = \int_{\Gm} h_{j} \pardrv{h_{k}}{n} \ds = p_{j,k}$ and $(p_{j,k})$ is again the period matrix of $\Om$, which is symmetric and invertible. Thus the condition on the $a_{k}$ is $\sum_{k=1}^{n}a_{k}p_{j,k} = \tht_{j}$ for $1 \leq j \leq n$, and again it can always be satisfied.
\end{proof}

%%%%%
\section{Beurling-Helson-Lowdenslager Theorem for $\Ltwo(\Gm,\om)$} \label{sec:LtwoBHL}

%\bl{$\Ltwo(\Gm,\om) = \Htwo(\Gm) \oplus \Htwo_{0}(\Gm)^{*} \oplus N(\Gm)$, $M_{\ps}$, $\Htwo(\Gm)$, $P$, $\cP(\ps)$, $\fG$}

In \cite[Theorem 1]{Forelli} (see also Fisher~\cite[Section 4.5]{Fisher}) Forelli proved that there exists a projection $\cP$ of $\Hinf(\T)$ onto $C_{\tau}\bigl( \Hinf(\Gm) \bigr)$ satisfying $\cP(fg) = f\cP(g)$ for all $f$ in $C_{\tau}\bigl( \Hinf(\Gm) \bigr)$ and $g$ in $\Hinf(\T)$. As Forelli noted, the projection is defined on $\Lone(\T,m)$ and maps it onto $C_{\tau}\bigl( \Lone(\Gm,\om) \bigr)$. It is obtained as follows. Let $P$ be the polynomial whose $n$ zeros are the critical points of the Green's function of $\Om$ with pole at $\what$. Since $\cE\bigl( \Hinf(\T) \bigr) = C_{\tau}\bigl( \Hinf(\Gm) + N(\Gm) \bigr)$, and $P\cdot \bigl( \Hinf(\Gm) + N(\Gm) \bigr) = \Hinf(\Gm)$ (see \cite{Forelli} and \cite{Fisher} again), it follows that $1/P \in \Hinf(\Gm) + N(\Gm)$ and there exists $\fp \in \Hinf(\T)$ such that $\cE(\fp) = C_{\tau}(1/P)$. Thus if $\cP$ is defined for $f \in \Lone(\T,m)$  by $\cP(f) = C_{\tau}(P)\cdot \cE(\fp f)$, then the following proposition holds (see \cite{Forelli}).

%%%
\begin{propns} \label{prop:Forelli}
\romannumbering
\begin{enumerate}
\item \label{prop:Forelli1}
$\cP\bigl(\Linf(\T) \bigr) = C_{\tau}\bigl( \Linf(\Gm) \bigr)$.
\item \label{prop:Forelli2}
$\cP\bigl(\Ltwo(\T) \bigr) = C_{\tau}\bigl( \Ltwo(\Gm) \bigr)$.
\item \label{prop:Forelli3}
For all $f \in \Ltwo(\Gm,\om)$, $\ps \in \Hinf(\T)$, $\cP(\ps C_{\tau}(f)) = \cP(\ps) C_{\tau}(f)$.
\item \label{prop:Forelli4}
$\cP\bigl(\Hinf(\T) \bigr) = C_{\tau}\bigl( \Hinf(\Gm) \bigr)$.
\item \label{prop:Forelli5}
$\cP\bigl(\Htwo(\T) \bigr) = C_{\tau}\bigl( \Htwo(\Gm) \bigr)$.
\end{enumerate}
\end{propns}

A basic idea for proving the Beurling-Helson-Lowdenslager theorem for $\Ltwo(\Gm,\om)$ is contained in the following lemma (see \cite{H}). We will write $M_{\ps}$ for the operator of multiplication by $\ps$.

%%%
\begin{lemns} \label{lem:Hasumi}
Suppose $\cM$ is a closed subspace of $\Ltwo(\Gm,\om)$ that is invariant under $M_{\ps}$ for every member $\ps$ of $\Hinf(\Gm)$ and 
\begin{equation} \label{eqn:defcN}
\cN = \clspan\set{\xi C_{\tau}(g): \xi \in \Hinf(\T) \text{ and } g \in \cM}.
\end{equation}
Then $\cN$ is a subspace of $\Ltwo(\T,m)$ that is invariant under multiplication by every function  in $\Hinf(\T)$, and $\cP(\cN) = C_{\tau}(\cM)$. If $\cM \subset \Htwo(\Gm)$, then $\cN \subset \Htwo(\T)$.
\end{lemns}

\begin{proof}
Suppose $\cM$ is a closed subspace of $\Ltwo(\Gm,\om)$ that is invariant under multiplication by every function in $\Hinf(\Gm)$, and 
let $\cN$ be defined by \eqref{eqn:defcN}.
Then $\cN$ is the smallest $\Hinf(\T)$-invariant subspace that includes $C_{\tau}(\cM)$. 
By Proposition~\ref{prop:Forelli}\eqref{prop:Forelli4}, Forelli's projection maps $\Hinf(\T)$ onto $C_{\tau}(\Hinf(\Gm))$, and thus it follows (Proposition~\ref{prop:Forelli}\eqref{prop:Forelli3}) that when $\cP$ is applied to a generator $\xi \cdot C_{\tau}(f)$ of $\cN$ the result is $\cP(\xi)\cdot C_{\tau}(f)$, which is a member of $C_{\tau}(\cM)$ because the first factor is in $C_{\tau}(\Hinf(\Gm))$, $C_{\tau}$ preserves products, and $\cM$ is invariant under multiplication by functions in $\Hinf(\Gm)$. Thus $\cP(\cN) \subset C_{\tau}(\cM)$. The opposite inclusion is true because, by Proposition~\ref{prop:Forelli}\eqref{prop:Forelli3}, $\cP$ fixes all members of $C_{\tau}\bigl( \Ltwo(\Gm,\om) \bigr)$, and hence we have $\cP(\cN) = C_{\tau}(\cM)$. The last assertion is immediate.
\end{proof}

%%%
\begin{thmns} \label{thm:BHL}
Let $\cM$ be a closed subspace of $\Ltwo(\Gm,\om)$ that is invariant under $M_{\ps}$ for every $\ps \in \Hinf(\Gm)$. 
Then either 
\romannumbering
\begin{enumerate}
\item 
$\cM = \chi_{\E}\Ltwo(\Gm,\om)$ for some measurable subset $\E$ of $\Gm$, 
or 
\item 
$\cM = \ph \Htwo(\Gm)$ for some $\ph \in \Linf(\Gm,\om)$ such that $\abs{\ph}$ is constant on each of the components of $\Gm$.
\end{enumerate}
In the case where $\cM \subset \Htwo(\Gm)$ we have that $\ph$ is a Royden inner function.
\end{thmns}
%%%

%%
\begin{proof}
Suppose $\cM$ and $\cN$ are as above. By Lemma~\ref{lem:Hasumi},  
 $\cP(\cN) = C_{\tau}(\cM)$. Since $\cN$ is a subspace of $\Ltwo(\T)$ that is invariant under multiplication by all functions in $\Hinf(\T)$, either (i) $\cN = \chi_{\F}\Ltwo(\T,m)$ where $\F$ is a measurable subset of $\T$, or else (ii) $\cN = q \Htwo(\T)$ where $q$ is a unimodular function on $\T$.
Thus it remains to analyze $\cP(\chi_{\F}\Ltwo(\T,m))$ in case one and $\cP(q\Htwo(\T))$ in case two. For this we use Lemma~\ref{lem:Ginv}, arguing the two cases separately.

Observe that the generators of $\cN$ are mapped into other generators of $\cN$ under composition with members of $\fG$, and thus $\cN$ is invariant under composition with elements of $\fG$. Therefore, if $\sig \in \fG$, then in case one $\chi_{\F}\of\sig = \chi_{\F}f$ for some $f$ in $\Ltwo(\T,m)$, and consequently $\chi_{\F}\of\sig \leq \chi_{\F}$. Since this holds for all members of the group $\fG$, it follows that $\chi_{\F}\of\sig = \chi_{\F}$. Thus there exists a measurable subset $\E$ of $\Gm$ such that $\chi_{\F} = \chi_{\E}\of\tau$. Thus we have
\begin{equation*}
C_{\tau}(\cM) = \cP(\cN) = (\chi_{\E}\of\tau)\cdot \cP\bigl( \Ltwo(\T,m) \bigr) = C_{\tau}(\chi_{E}\cdot\Ltwo(\Gm)),
\end{equation*} 
which implies $\cM = \chi_{\E}\cdot\Ltwo(\Gm,\om)$, thereby completing the proof in this case.

In the second case, invariance of $\cN$ under composition with members of $\fG$ leads to $q\of\sig = q f_{\sig}$ for some $f_{\sig} \in \Htwo(\T)$. Since $q\of\sig$ and $q$ are unimodular, it follows that $f_{\sig}$ is inner. Also $q\of\sig\inv = q f_{\sig\inv}$, and hence on composing with $\sig$ we obtain $q = (qf_{\sig})\cdot(f_{\sig\inv}\of\sig)$, which implies $1 = f_{\sig}\cdot (f_{\sig\inv}\of\sig)$. Thus the inner function $f_{\sig}$ has an inverse in $\Hinf(\T)$ and is therefore a constant in $\T$. Call it $\eta(\sig)$, so $q\of\sig = \eta(\sig)q$. Again it is clear that $\eta(\sig) = 1$, when $\sig(z) = z$, the identity in $\fG$. Also we have $\eta(\sig_{1}\of\sig_{2})q = (q\of\sig_{1})\of\sig_{2} = \eta(\sig_{1}) \eta(\sig_{2})q$, and it follows that $\eta$ is a homomorphism of $\fG$ into $\T$. By Lemma~\ref{lem:eigenfneta}, there exists an invertible function $F$ in $\Hinf(\T)$, with $\abs{F}$ constant on the sets $\tau\inv(\Gm_{j})$ for $0 \leq j \leq n$, and satisfying $F\of\sig = \eta(\sig)F$ for all $\sig \in \fG$. It follows that $q/F$ is unchanged by composition with members of $\fG$, and therefore Lemma~\ref{lem:Ginv} implies $q/F = \ph\of\tau$ for some $\ph \in \Linf(\Gm,m)$. Since $q$ is unimodular and $\abs{F}$ is constant on each set $\tau\inv(\Gm_{j})$, it follows that $\abs{\ph}$ is constant on each of the sets $\Gm_{j}$.

Since $F\Hinf(\T) = \Hinf(\T)$, we have $\cN = (\ph\of\tau)\cdot \Htwo(\T)$. As in the first case we now have 
\begin{equation*}
C_{\tau}(\cM) = \cP(\cN) = (\ph\of\tau)\cdot \cP(\Htwo(\T)) = C_{\tau}(\ph\cdot\Htwo(\Gm)),
\end{equation*}
which implies $\cM = \ph\cdot\Htwo(\Gm)$ as required.

If $\cM \subset \Htwo(\Gm)$, case (i) can not occur, and consequently $\ph \in \Htwo(\Gm)$. Since $\abs{\ph}$ is constant on the connected components of $\Gm$, $\ph$ is a Royden inner function. This completes the proof.
\end{proof}

The question of uniqueness of the representation will be addressed in Section~\ref{sec:inout} (see Theorem~\ref{thm:uniqueness}). In the following we identify the cyclic vectors for $\Htwo(\Gm)$ as the outer functions. The general case of this result will also be obtained in Section~\ref{sec:inout} (see Theorem~\ref{thm:bnddcyclicouter}).

\begin{thmns} \label{thm:outercyclicH2}
A function $f \in \Htwo(\Gm)$ is cyclic in the sense that $\Hinf(\Gm)\cdot f$ is dense in $\Htwo(\Gm)$ if and only if $f$ is outer.
\end{thmns}

\begin{proof}
If $f$ is cyclic, then the invariant subspace $\cM$ it generates is all of $\Htwo(\Gm)$. In this case the invariant subspace $\cN$ of Lemma~\ref{lem:Hasumi} is generated by $f\of\tau$ and is also all of $\Htwo(\T)$. Thus $f\of \tau$ is outer in $\Htwo(\T)$, and consequently Lemma~\ref{lem:outer} implies that $f$ is outer in $\Htwo(\Gm)$. Conversely, if $f$ is outer in $\Htwo(\Gm)$, then Lemma~\ref{lem:outer} implies that $f \of \tau$, which generates $\cN$, is outer in $\Htwo(\T)$, and thus $\cN = \Htwo(\T)$. This in turn implies that $\cM = \Htwo(\Gm)$, and consequently $f$ is cyclic.
\end{proof}

%%%%%
\section{Beurling-Helson-Lowdenslager Theorem for $L^{\al}(\Gm,\om)$}

Throughout this section $\al$ will be a continuous, dominating, normalized gauge norm on $\Linf(\Gm,\om)$, i.e.\ $\al \in \fN$. To generalize Theorem~\ref{thm:BHL} to the spaces $\Lal(\Gm,\om)$, we use the same technique as that of the first author in \cite{Chen2} with a few modifications necessitated by the multiple connectedness of the domain of the members of $\Hal(\Om)$. In that paper invariant subspaces of the single operator multiplication by $z$ on $\Lal(\T,m)$ were considered, in which case invariance under that operator is enough to imply invariance under multiplication by all $\Hinf(\T)$ functions. In the multiply connected case, the invariant subspaces of the operator multiplication by $z$ are more complicated (see \cite{Hitt, AR1, AR2}), and so we  assume invariance under multiplication by all $\Hinf(\Gm)$ functions. A basic idea in \cite{Chen2} was also devised earlier by Gamelin~\cite{G1, G2} to study invariant subspaces in certain generalized $\Hp$ spaces. 

Let $\B$ be the closed unit ball of $\Linf(\Gm,\om)$. The next lemma is Lemma 2.9 of \cite{Chen2}.

\begin{lemns} \label{lem:ballLinf}
If $\al \in \fN$, then
\romannumbering
\begin{enumerate}
\item 
on $\B$ the $\al$-topology, the $\twonorm{\cdot}$-topology, and the topology of convergence in measure coincide, and
\item 
$\bB$ is $\al$-closed.
\end{enumerate}
\end{lemns}

%%%
\begin{lemns} \label{lem:HoneLal}
$\Hal(\Gm) = \Hone(\Gm) \cap \Lal(\Gm,\om).$
\end{lemns}
%%%

\begin{proof}
The inclusion of $\Hal(\Gm)$ in $\Hone(\Gm)$ is a consequence of the dominating property, and its inclusion in the intersection follows. For the opposite inclusion, suppose $\ph \in \Lal(\Gm,\om)^{\#}$ is in the annihilator of $\Hal(\Gm)$. By Proposition~\ref{prop:dualLal}, there exists $F \in \cL^{\al'}(\Gm,\om)$ such that for all $f \in \Lal(\Gm,\om)$, $fF \in \Lone(\Gm,\om)$ and $\ph(f) = \int_{\Gm}fF \domega$.

Because $\ph$ is in the annihilator of $\Hal(\Gm)$ we have $\int_{\Gm} fF \domega = 0$ for all $f \in \Hinf(\Gm)$, and it follows from Theorem 4.8 of \cite{Fisher} that $PF \in \Hone(\Gm)$, where $P$ is the polynomial whose zeros are the critical points of the Green's function of $\Om$ with pole at $\what$. Further, because $P\cdot (\Hone(\Gm) + N(\Gm)) = \Hone(\Gm)$, it follows that $F \in \Hone(\Gm) + N(\Gm)$, and thus $F = F_{1} + F_{N}$, where $F_{1} \in \Hone(\Gm)$ and $F_{N} \in N(\Gm)$. Since $1 \in \Hinf(\Gm)$ and $\int_{\Gm} f \domega = 0$ for all $f \in N(\Gm)$, $\int_{\Gm} 1 F_{1} \domega = \int_{\Gm} 1 F \domega = 0$, and thus $F_{1} \in \Hone_{0}(\Gm)$.

Suppose $g \in \Hone(\Gm) \cap \Lal(\Gm,\om)$. Then $gF \in \Lone(\Gm,\om)$, and because $F_{N}$ is bounded, $gF_{1} \in \Lone(\Gm,\om)$, and consequently $C_{\tau}(gF_{1}) \in \Lone(\T,m)$. Also, from $g \in \Hone(\Gm)$ and $F_{1} \in \Hone_{0}(\Gm)$, it follows that $C_{\tau}(g) \in \Hone(\T)$ and $C_{\tau}(F_{1}) \in \Hone_{0}(\T)$. Thus the product of $C_{\tau}(g)$ and $C_{\tau}(F_{1})$ is in $\Hone_{0}(\T)$, which implies $\int_{\Gm} gF_{1} \domega = \intT (g F_{1})\of \tau \dm = 0$. Since $\Htwo(\Gm)\adj$ and $N(\Gm)$ are orthogonal in $\Ltwo(\Gm,\om)$, and since $\Htwo(\Gm)$ is dense in $\Hone(\Gm)$, it follows that $\int_{\Gm} g F_{N}\domega = 0$. Consequently $\ph(g) = 0$, and the Hahn-Banach theorem now implies that $g \in \Hal(\Gm)$, thereby giving us the required opposite inclusion.
\end{proof}

The next lemma is fundamental for what follows.

%%%
\begin{lemns} \label{lem:Saito}
If $b \in \Linf(\Gm,\om)$ and $1/b \in \Lal(\Gm,\om)$, then there exists a function $\ps$ having $\om$-a.e.\ constant modulus on each connected component of $\Gm$, and there exists an outer function $h \in \Hinf(\Gm)$ such that $b = \ps h$ and $1/h \in \Hal(\Gm)$.
\end{lemns}

\begin{proof} 
If $b$ satisfies the hypothesis, then, since $\Lal(\Gm,\om) \subset \Lone(\Gm,\om)$, it follows that $\log\abs{b}$ is integrable, and hence there exists a harmonic function $u$ on $\Om$ with $\log\abs{b}$ as its boundary function. By Lemma~\ref{lem:Per}, there exists a harmonic unit $u_{0}$ such that $u-u_{0}$ has a harmonic conjugate function $v$ on $\Om$.  Put $h = \exp(u-u_{0} + iv)$ to get an outer function on $\Om$ such that $\abs{h}$ has a boundary function $\abs{b} e^{-u_{0}}$, and thus $h \in \Hinf(\Gm)$. If $\ps = b/h$, then $\abs{\ps} = e^{u_{0}}$ which is constant on each of the sets $\Gm_{j}$. Finally, $1/h$ is in both $\Hone(\Gm)$ and $\Lal(\Gm,\om)$, and thus Lemma~\ref{lem:HoneLal} implies $1/h \in \Hal(\Gm)$.
\end{proof}

%%%
\begin{propns} \label{prop:clwkstrtostr}
Let $M$ be a weak* closed subspace of $\Linf(\Gm,\om)$ that is invariant under multiplication by all members of $\Hinf(\Gm)$. If $\cM$ is the closure of $M$ in $\Lal(\Gm,\om)$, then $\cM$ is also invariant under multiplication by members of $\Hinf(\Gm)$ and $M = \cM \cap \Linf(\Gm,\om)$.
\end{propns}

\begin{proof}
That $\cM$ is invariant and $M \subset \cM \cap \Linf(\Gm,\om)$ are immediate. For the opposite inclusion we show that if $\ph$ is a weak* continuous linear functional on $\Linf(\Gm,\om)$ with $M$ in its kernel, then $\ph(g) = 0$ for every $g \in \cM \cap \Linf(\Gm,\om)$. By the Hahn-Banach theorem, this will imply that $\cM \cap \Linf(\Gm,\om) \subset M$. For this suppose $F \in \Lone(\Gm,\om)$ and $\int_{\Gm} fF \domega = 0$ for every $f \in M$. Apply Lemma~\ref{lem:Saito} with $b = 1/(\abs{F} + 1)$ to get an outer function $h \in \Hinf(\Gm)$ such that $1/h \in \Hone(\Gm)$ and a function $\ps$ that has constant modulus on each of the sets $\Gm_{j}$ for $0 \leq j \leq n$ satisfying $b = \ps h$. Since 
\begin{equation*}
\abs{hF} = \frac{1}{\abs{\ps}} \frac{\abs{F}}{\abs{F} + 1},
\end{equation*}
$hF \in \Linf(\Gm,\om)$.
There exists a sequence $\seq{h_{\nu}}$ in $\Hinf(\Gm)$ such that $\lim_{\nu\ra\infty}\onenorm{1/h - h_{\nu}} = 0$, and thus $\onenorm{F - h_{\nu}hF} = \onenorm{(1/h - h_{\nu})hF} \leq \onenorm{1/h - h_{\nu}} \supnorm{hF} \ra 0$ as $\nu\ra \infty$.

If $g_{\mu} \in M$, then $h_{\nu}h g_{\mu} \in M$, and it follows that $\int_{\Gm}  h_{\nu}hg_{\mu}F \domega = 0$. If $g \in \cM$, then there is a sequence $\seq{g_{\mu}}$ in $M$ such that $\al(g - g_{\mu}) \ra 0$ as $\mu \ra \infty$, which implies $\onenorm{g - g_{\mu}} \ra 0$, since $\al$ is dominating. Thus 
\begin{equation}
\Bigl| \int_{\Gm} g h_{\nu} h F \domega \Bigr| = \Bigl| \int_{\Gm} (g - g_{\mu}) h_{\nu} h F \domega \Bigr| \leq \onenorm{g - g_{\mu}} \supnorm{h_{\nu}hF} \ra 0
\end{equation}
as $\mu \ra \infty$, and consequently $\int_{\Gm} g h_{\nu} h F \domega = 0$ for every $\nu$.

Finally if in addition to $g \in \cM$ we also assume that $g \in \Linf(\Gm,\om)$, then we have that 
\begin{equation*}
\Bigl| \int_{\Gm} g F \domega \Bigr| = \Bigl| \int_{\Gm} g (F - h_{\nu} h F) \domega \Bigr| \leq \supnorm{g} \onenorm{F - h_{\nu} hF} \ra 0
\end{equation*}
as $\nu \ra \infty$. Thus $\int_{\Gm} g F \domega = 0$, which completes the proof.
\end{proof}

%%%
\begin{propns} \label{prop:clstrtowkstr}
Let $\cM$ be a closed subspace of $\Lal(\Gm,\om)$ that is invariant under multiplication by all members of $\Hinf(\Gm)$. If $M = \cM \cap \Linf(\Gm,\om)$, then $M$ is weak* closed and invariant and $\cM = M^{-\al}$.
\end{propns}

\begin{proof}
That $M$ is weak* closed follows from the Krein-\v{S}mulian theorem and Lemma~\ref{lem:ballLinf} as in \cite{Chen2}. Invariance is immediate.

It is clear that $M^{-\al} \subset \cM$. Consider $f \in \cM$, and apply Lemma~\ref{lem:Saito} to $b = 1/(\abs{f} + 1)$, thereby producing a function $\ps$ with $\om$-a.e.\ constant modulus on each component of $\Gm$ and an outer function $h \in \Hinf(\Gm)$ with $1/h \in \Hal(\Gm)$ such that $b = \ps h$. There exists a sequence $\seq{h_{\nu}}$ in $\Hinf(\Gm)$ such that $\al(1/h - h_{\nu}) \ra 0$ as $\nu \ra \infty$. Since $\abs{hf} = \abs{\ps} \abs{f}/(\abs{f} + 1)$, it follows that $hf \in \cM$ and $hf$ is bounded, and hence $hf \in M$. The same is true of each $h_{\nu}hf$, and $\al(f - h_{\nu}hf) \leq \al(1/h - h_{\nu}) \supnorm{hf} \ra 0$ as $\nu \ra \infty$.  Therefore $f \in M^{-\al}$, and the proof is complete\end{proof}

With Propositions~\ref{prop:clwkstrtostr} and \ref{prop:clstrtowkstr} in hand we can now prove the principal result, the Beurling, Helson-Lowdenslager theorem for a space with a continuous, dominating, normalized, gauge norm on a multiply connected domain. As mentioned in the introduction, the last statement contains Royden's version of Beurling's theorem in \cite[Theorem 1]{R}.

%%%
\begin{thmns} \label{thm:BHLal}
Let $\cM$ be a closed subspace of $\Lal(\Gm,\om)$ that is invariant under $M_{\ps}$ for every $\ps \in \Hinf(\Gm)$.
Then either 
\romannumbering
\begin{enumerate}
\item 
$\cM = \chi_{\E}\Lal(\Gm,\om)$ for some measurable subset $\E$ of $\Gm$, 
or 
\item 
$\cM = \ph \Hal(\Gm)$ for some $\ph \in \Linf(\Gm,\om)$ such that $\abs{\ph}$ is constant on each of the components of $\Gm$.
\end{enumerate}
The result is also true in the case where $\al$ is the essential supremum norm when $\cM$ is weak* closed. When $\cM \subset \Hal(\Gm)$ case (ii) holds and the function $\ph$ is a Royden inner function.

\end{thmns}

\begin{proof}
The case of $\Ltwo(\Gm,\om)$ was handled in Theorem~\ref{thm:BHL}. Suppose $M$ is a weak* closed subspace of $\Linf(\Gm,\om)$ that is invariant under $M_{\ps}$ for every $\ps \in \Hinf(\Gm)$, and let $\cM$ be the closure of $M$ in $\Ltwo(\Gm,\om)$. The preceding case then applies to $\cM$, and Proposition~\ref{prop:clwkstrtostr} implies that $M$ is obtained by intersecting $\cM$ with $\Linf(\Gm,\om)$. Since the intersection of $\chi_{\E}\Ltwo(\Gm,\om)$ with $\Linf(\Gm,\om)$ is $\chi_{\E}\Linf(\Gm,\om)$ and the intersection of $\ph\Htwo(\Gm,\om)$ with $\Linf(\Gm,\om)$ is $\ph\Hinf(\Gm,\om)$, this case is proved.

Next let $\cM$ be a closed subspace of $\Lal(\Gm,\om)$ for $\al \in \fN$. By Proposition~\ref{prop:clstrtowkstr}, if $M = \cM \cap \Linf(\Gm,\om)$, then $M$ is weak* closed and invariant under each $M_{\ps}$ with $\ps \in \Hinf(\Gm)$. The preceding case now implies that either $M = \chi_{\E}\Linf(\Gm,\om)$ or $M = \ph \Hinf(\Gm)$. The closure of $M$ in the $\al$ topology is $\cM$, by Proposition~\ref{prop:clstrtowkstr}, the $\al$ closure of $\chi_{\E}\Linf(\Gm,\om)$ is $\chi_{\E}\Lal(\Gm,\om)$, and the $\al$ closure of $\ph \Hinf(\Gm)$ is $\ph \Hal(\Gm)$. 

The final assertion is clear and thus the proof is complete.
\end{proof}

%%%%%
\section{Inner Outer Factorization} \label{sec:inout}

%\rd{ singular inner functions???}

As mentioned previously, the Royden definition makes a function inner if it is in $\Hinf(\Om)$ and its boundary function has $\om$-a.e.\ constant absolute values on each connected component of $\Gm$. Let $I(\Om)$ be the multiplicative semigroup of inner functions on $\Om$ and $I\inv(\Om)$ the subgroup of invertible ones. As usual, 
$I(\Gm)$, $I\inv(\Gm)$ will be the sets of their boundary functions on $\Gm$.
The Royden definition of inner function allows one to describe all of the invariant subspaces of the $\Hinf(\Gm)$ multiplication operators on $\Hal(\Gm)$, including $\set{0}$, very simply. They are the subspaces of the form $\ph \Hal(\Gm)$, where $\ph$ is inner, but the correspondence is not one-to-one. It is easy to disregard this phenomenon on the disk since if two inner functions produce the same invariant subspace of $\Hal(\T)$, then they differ by at most a multiplicative constant of modulus one. 

The multiply connected case is more complicated than that of the disk. A minor difference is that with Royden's definition the constant 0 function is inner. Also, a nonzero multiple of a non zero inner function is inner and gives rise to the same invariant subspace. A modification of this situation, making it a bit more in line with the situation on the unit disk, can be made by defining a \emph{normalized inner function} as one whose absolute boundary values on the outer boundary $\Gm_{0}$ of $\Om$ are one $\om$-a.e., and from here on inner functions on $\Om$ other than the trivial one 0 will be assumed to have this normalization. Thus  a Royden inner function, or as we will say from here on, simply \emph{inner} function, on $\Om$ will be either the constant 0 function or a bounded analytic function $\ph$ on $\Om$ such that $\abs{\ph}$ has boundary values $\om$-a.e.\ equal to 1 on the outer boundary $\Gm_{0}$ and $\om$-a.e.\ (necessarily nonzero) constant values on each of the remaining connected boundary components $\Gm_{j}$, $1 \leq j \leq n$. Equivalently, a nonzero inner function $\ph$ is a member of $\Hinf(\Gm)$ that satisfies $\log\abs{\ph} = \sum_{j=1}^{n} a_{j} \chi_{\Gm_{j}}$ for some real constants $a_{j}$, $1\leq j \leq n$.

But a more fundamental difference is that there are nonconstant invertible inner functions. For example, on a zero centered annulus with outer radius 1 the functions $w^{k}$ with $k \in \Z$ are all normalized inner and invertible. Kuratowski in \cite{K} (see also \cite[page 211]{SZ}) studied functions without zeros on $\Om$. He showed that if points $a_{1}, a_{2}, \ldots, a_{n}$ are chosen in the complement of $\Ombar$ with one of them in each bounded component of that complement, and if 
\begin{equation} \label{eqn:frational}
f(w) = (w - a_{1})^{k_{1}}\cdots (w - a_{n})^{k_{n}},  
\end{equation}
where  $k_{1}, k_{2}, \ldots, k_{n}$ are integers, then every analytic function $\ph$ on $\Om$ having no zeros is homotopic to a function of the form $f$. This means that $\ph(w) = f(w)e^{g(w)}$, where $g$ is analytic on $\Om$. Moreover the $k_{j}$'s are uniquely determined by $\ph$, and $g$ is uniquely determined up to the addition of an integer multiple of $2\pi i$. Moreover, $\ph$ is bounded if and only if $\Real g$ is bounded above, and $\ph$ is invertible in $\Hinf(\Om)$ if and only if $\Real g$ is bounded both above and below. It follows that $\ph$ is an invertible inner function precisely when, in addition, there exists a harmonic unit $u$ such that $u = \Real g + \log\abs{f}$.

%%%
\begin{propns} \label{prop:invinner}
Let $\ph$ be a nonzero inner function in $\Hinf(\Gm)$. Then the following are equivalent:
\romannumbering
\begin{enumerate}
\item 
$\ph$ is invertible,
\item 
$\ph$ is outer,
\item 
for some $\al \in \fN_{\infty}$, $\ph \Hal(\Gm) = \Hal(\Gm)$,
\item 
for all $\al \in \fN_{\infty}$, $\ph \Hal(\Gm) = \Hal(\Gm)$.
\item
there exist integers $k_{1}, k_{2},\ldots , k_{n}$ inducing $f$ as in \eqref{eqn:frational} and there exists $g$ in $H(\Om)$ with bounded real part such that $\log\abs{f} + \Real g$ is a harmonic unit and $\ph = f e^{g}.$
\end{enumerate}
\end{propns}

%(It can also be shown that an inner function $\ph$ is invertible if and only if there is a harmonic unit $u$ on $\Om$ that has a harmonic conjugate $v$ with periods that are integer multiples of $2\pi$ such that $\ph = \exp(u+iv)$.)

\begin{proof}
It is easy to see that (i) is equivalent to each of (iii) and (iv). Also (iv) implies that $\ph$ is cyclic for $\Htwo(\Gm)$ which, by Theorem~\ref{thm:outercyclicH2}, implies $\ph$ is outer. Finally, if $\ph$ is outer, then Theorem~\ref{thm:outercyclicH2} implies $\ph\Htwo(\Gm)$ is dense in $\Htwo(\Gm)$. Since $\ph$ is inner, it is bounded below on $\Gm$, and thus $\ph\Htwo(\Gm)$ is closed, hence all of $\Htwo(\Gm)$. Therefore $\ph$ is invertible. Equivalence of (i) and (v) follows from the discussion preceding the statement of the proposition.
\end{proof}

We will call two inner functions \emph{equivalent} if they differ by an invertible factor, which is necessarily inner. Modulo equivalence, $I(\Om)$ has a lattice structure given by divisibility. For $\ph, \ps \in I(\Om)$ we say $\ph$ \emph{divides} $\ps$, written 
\begin{equation*}
\ph \mid \ps,
\end{equation*}
if and only if there is a $\rho \in I(\Om)$ such that $\ps = \rho \ph$. Since $I(\Om)$ is cancellative, $\rho$ is unique up to equivalence. Also, $\phi$ and $\psi$ are equivalent if each divides the other.

%%%
\begin{thmns} \label{thm:uniqueness}
Suppose $\ph$ and $\ps$ are inner functions on $\Gm$. Each of the following implies the others.
\romannumbering
\begin{enumerate}
\item 
$\ph$ and $\ps$ are equivalent,
\item 
$\ph \Hone(\Gm) = \ps \Hone(\Gm)$,
\item 
$\ph \Hinf(\Gm) = \ps \Hinf(\Gm)$,
\item 
for some $\al \in \fN_{\infty}$, $\ph \Hal(\Gm) = \ps \Hal(\Gm)$,
\item 
for every $\al \in \fN_{\infty}$, $\ph \Hal(\Gm) = \ps \Hal(\Gm)$.
\end{enumerate}
\end{thmns}

\begin{proof}
The last four cases follow from the first because if $\rho$ is an invertible inner function, then $\rho \Hal(\Gm) = \Hal(\Gm)$ for all $\al$. In each of the last four cases $\ph$ is a multiple of $\ps$ and vice versa, which makes $\ph$ and $\ps$ equivalent.
\end{proof}

Suppose $\emp \ne S \subset I(\Om)$. We say that $\rho \in I(\Om)$ is a \emph{least common multiple} of $S$ if and only if $\rho$ is divisible by every element of $S$ and if $\sig \in I(\Om)$ is divisible by every element of $S$, then $\sig \mid \rho$. It is clear that if there is such a $\rho$, then it is unique modulo equivalence. Although uniqueness is only for equivalence classes, we write 
\begin{equation*}
\rho = \LCM(S)
\end{equation*} 
to denote that $\rho$ is a least common multiple of $S$.

Similarly we say that $\rho$ is a (unique up to equivalence) \emph{greatest common divisor} of $S$, denoted by $\rho = \GCD(S)$, if and only if $\rho$ divides every member of $S$ and every $\sig \in I(\Om)$ that divides every member of $S$ must also divide $\rho$. The proof of the following proposition follows immediately from Theorem~\ref{thm:BHLal} and the fact that $\bigcap_{\ps\in S} \ps\Hone(\Om)$ and $\clspan^{\Hone(\Om)}\bigl(\bigcup_{\ps\in S} \ps \Hone(\Om)\bigr)$ are always closed $\Hinf(\Om)$-invariant subspaces of $\Hone(\Om)$.

%%%
\begin{propns}
Suppose $\emp \ne S \subset I(\Om)$ and $\ph, \gm \in I(S)$. Then
\romannumbering
\begin{enumerate}
\item 
$\gm \mid \ph$ if and only if $\ph\Hone(\Om) \subset \gm \Hone(\Om)$,
\item 
$\gm$ and $\ph$ are equivalent if and only  if each divides the other,
\item 
$\gm = \LCM(S)$ if and only if $\gm \Hone(\Om) = \bigcap_{\ps \in S} \ps \Hone(\Om)$,
\item 
$\gm = \GCD(S)$ if and only if $\gm \Hone(\Om) = \clspan^{\Hone(\Om)}\bigl(\bigcup_{\ps \in S} \ps \Hone(\Om) \bigr)$,
\item 
$\LCM(S)$ always exists; $\GCD(S)$ exists if and only if $\bigcap_{\ps \in S} \ps \Hone(\Om) \ne \set{0}$. In particular, if $S$ is finite, then $\GCD(S)$ exists.
\end{enumerate}
\end{propns}

The inner-outer factorization of elements of $\Hone(\Gm)$, and hence of elements of every $\Hal(\Gm)$, could have been obtained earlier by transferring $f \in \Hone(\Gm)$ to $\Hone(\T)$ via $C_{\tau}$, but we will obtain both the factorization and the characterization of cyclic vectors as outer functions here as consequences of the Beurling theorem in this setting, Theorem~\ref{thm:BHLal}. Here a vector $f$ is cyclic for a space $\Hal(\Gm)$ means $\Hinf(\Gm)\cdot f$ is norm dense, respectively weak$\adj$ dense if $\al = \supnorm{\cdot}$, in $\Hal(\Gm)$. 

%%%
\begin{thmns}\label{thm:bnddcyclicouter}
Suppose $f \in \Hinf(\Gm)$. The following are equivalent:
\romannumbering
\begin{enumerate}
\item \label{item:outer}
$f$ is outer,
\item \label{item:cyctwo}
$f$ is cyclic for $\Htwo(\Gm)$,
\item \label{item:cycHinf}
$f$ is cyclic for $\Hinf(\Gm)$,
\item \label{item:cycallal}
$f$ is cyclic for all $\Hal(\Gm)$ with $\al \in \fN$,
\item \label{item:cycsomeal}
$f$ is cyclic for some $\Hal(\Gm)$ with $\al \in \fN$.
\end{enumerate}
\end{thmns}

\begin{proof}
\eqref{item:outer} is equivalent to \eqref{item:cyctwo}. This is implied by Theorem~\ref{thm:outercyclicH2}.

\eqref{item:cyctwo} implies \eqref{item:cycHinf}. Suppose $f$ is cyclic for $\Htwo(\Gm)$ and $M = (\Hinf(\Gm)\cdot f)^{-\text{w}\adj}$, where the last superscript designates the weak$\adj$ closure. If $\cM = M^{-\twonorm{\cdot}}$, then $\cM$ is a $\twonorm{\cdot}$-closed invariant subspace of $\Htwo(\Gm)$, and Theorem~\ref{thm:BHL} implies $\cM = \ph\Htwo(\Gm)$ for some inner function $\ph$. Since $f$ is cyclic and belongs to $\cM$, $\Htwo(\Gm) \subset \cM$, and thus $\cM = \ph \Htwo(\Gm) = \Htwo(\Gm)$. Proposition~\ref{prop:clwkstrtostr} implies $M = \cM \cap \Linf(\Gm,\om) = \Hinf(\Gm)$, so $f$ is cyclic for $\Hinf(\Gm)$.

\eqref{item:cycHinf} implies \eqref{item:cycallal}. Suppose $f$ is cyclic for $\Hinf(\Gm)$ and $\al \in \fN$. Put $\cM = (\Hinf(\Gm) \cdot f)^{-\al}$ and $M = \cM \cap \Linf(\Gm,\om)$. By Proposition~\ref{prop:clstrtowkstr}, $M$ is weak$\adj$ closed and $M^{-\al} = \cM$. Since $f \in M$ and $f$ is cyclic for $\Hinf(\Gm)$, $M = \Hinf(\Gm)$, and thus $\cM \supset \Hinf(\Gm)$. It follows that $\cM = \Hal(\Gm)$ and $f$ is cyclic for $\Hal(\Gm)$.

Trivially \eqref{item:cycallal} implies \eqref{item:cycsomeal}, so 
we can complete the proof by showing that \eqref{item:cycsomeal} implies \eqref{item:cycHinf}.
Suppose $f$ is cyclic for some $\Hal(\Gm)$. Put $M = (\Hinf(\Gm)\cdot f)^{-\al}$ and $\cM = M^{-\al}$. Since $\cM$ is a closed invariant subspace of $\Hal(\Gm)$ that contains the cyclic vector $f$, $\cM = \Hal(\Gm)$. Proposition~\ref{prop:clwkstrtostr} implies $M = \cM \cap \Linf(\Gm,\om) = \Hinf(\Gm)$, so $f$ is cyclic for $\Hinf(\Gm)$.
\end{proof}

The cyclicity condition for bounded functions of the preceding theorem can now be expanded to hold for all functions in $\Hal(\Gm)$ for arbitrary $\al$.

%%%
\begin{corns} \label{cor:cycouter}
A function $f \in \Hal(\Gm)$ is cyclic if and only if $f$ is outer.
\end{corns} 

\begin{proof}
Let $f$ be any vector in $\Hal(\Gm)$ for arbitrary $\al \in \fN$, and let $\cM$ be the cyclic subspace of $\Hal(\Gm)$ generated by $f$, i.e.\ $\cM = (\Hinf(\Gm)\cdot f)^{-\al}$. It is clear that if $h$ is any function in $\Hinf(\Gm)$, then the cyclic subspace $\cN$ generated by $hf$ is included in $\cM$. We will show that $h$ can be chosen to be outer so that the reverse inclusion holds.

Put $g = \frac{1}{\abs{f} + 1}$. Thus $g \in \Linf(\Gm,\om)$ and $\frac{1}{g} \in \Lal(\Gm,\om)$. Lemma~\ref{lem:Saito} implies the existence of a function $\ps$ with constant modulus on each connected component of $\Gm$ and an outer function $h$ in $\Hinf(\Gm)$ such that $\frac{1}{h} \in \Hal(\Gm)$ and $g = \ps h$. Then $hf \in \cM$ and there exist $h_{k} \in \Hinf(\Gm)$ such that $\al(\frac{1}{h} - h_{k}) \ra 0$ as $k \ra \infty$. Because $\abs{hf} = \frac{g}{\abs{\ps}} \abs{f} \leq \frac{1}{\abs{\ps}}$, $hf$ is bounded and hence $\al(f - h_{k}hf) = \al(hf(\frac{1}{h} - h_{k})) \ra 0$. It follows that the cyclic subspace generated by  $hf$ contains $f$ and therefore $\cM$ is included in $\cN$.

We have shown that the cyclic subspace generated by $f$ is also generated by the bounded function $hf$. Thus $f$ in $\Hal(\Gm)$ is a cyclic vector for $\Hinf(\Gm)$ if and only if the bounded function $hf$ is, and Theorem~\ref{thm:bnddcyclicouter} therefore implies $f$ is cyclic if and only if $hf$ is outer. But $h$ is outer; it follows that $hf$ is outer if and only if $f$ is, and hence $f$ is cyclic if and only if it  is  outer.
\end{proof}

%%%
\begin{corns}
Every function $f$ in $\Hone(\Gm)$ has a factorization $f = \ph g$ where $\ph$ is inner and $g$ is outer. The  factors are unique up to equivalence.
\end{corns}

\begin{proof}
If $f \in \Hone(\Gm)$, let $\cM$ be the cyclic subspace generated by $f$. This is a closed invariant subspace, and Theorem~\ref{thm:BHLal} implies $\cM = \ph \Hone(\Gm)$ for some inner function $\ph$. Then $f = \ph g$ for some $g \in \Hone(\Gm)$. If $h \in \Hone(\Gm)$, then $\ph h \in \cM$, and consequently there is a sequence of vectors $h_{k}$ in $\Hinf(\Gm)$ such that $\onenorm{h_{k}f - \ph h} \ra 0$ as $k \ra \infty$, and this implies $\onenorm{h_{k}g - h} \ra 0$. Thus $g$ is a cyclic vector for $\Hone(\Gm)$, and, by Corollary~\ref{cor:cycouter}, $g$ is outer.

If $f = \ps h$ is a second inner-outer factorization of $f$, then the relation $\frac{g}{h} = \frac{\ps}{\ph}$ shows that $\frac{g}{h}$ is an outer function having boundary values that have constant absolute values on each connected component of $\Gm$. Thus $\ps = \frac{g}{h} \ph$, and $\frac{g}{h}$ is an invertible inner function by Proposition~\ref{prop:invinner}, which makes $\ph$ and $\ps$ equivalent.
\end{proof}

Suppose $0 \ne f \in \Hone(\Om)$. We define the \emph{zero set for} $f$ to be 
\begin{equation*}
\cZ(f) = \set{w \in \Om: f(w) = 0}.
\end{equation*}
There is also a multiplicity function $m_{f}:\cZ(f) \ra \N$ defined by $m_{f}(a)$ is the order of $a$ as a zero of $f$. Suppose $\K$ is a subset of $\Om$ without limit points in $\Om$ and $\nu:\K \ra \N$. We can define a closed $\Hinf(\Gm)$-invariant subspace $M(\K,\nu)$ to be the set of all $f \in \Hone(\Om)$ such that $\K \subset \cZ(f)$, and for every $a \in \K$, 
\begin{equation*}
m_{f}(a) \geq \nu(a).
\end{equation*}

\begin{propns}
Suppose $\K \subset \Om$ has no limit points in $\Om$ and $\nu:\K \ra \N$. There are only two possibilities:
\romannumbering
\begin{enumerate}
\item 
$M(\K,\nu) = \set{0}$, which means that $(\K,\nu)$ is determining for $\Hone(\Om)$, i.e.\ if $f,g \in \Hone(\Om)$ and $f \mid_{\K} = g\mid_{\K}$ and $m_{f-g} \geq \nu$ on $\K$, then $f = g$. If $\nu$ is the constant 1 on $\K$, this means that $\K$ is determining for $\Hone(\Om)$, i.e.\ $f = g$ on $\K$ implies $f = g$.
\item 
there is an inner function $\ph$ such that 
\begin{enumerate}
\item 
$M(\K,\nu) = \ph \Hone(\Om)$,
\item 
$\cZ(\ph) = \K$, and 
\item 
$m_{\ph} = \nu$.
\end{enumerate}
\end{enumerate}
\end{propns}

\begin{proof}
Suppose (i) is false and $M(\K,\nu) \ne \set{0}$. It follows from Theorem~\ref{thm:BHLal} that there is an inner function $\ph$ such that $M(\K,\nu) = \ph \Hone(\Om)$. Since $\ph = \ph \cdot 1 \in M(\K,\nu)$, we know that $\K \subset \cZ(\ph)$ and $\nu \leq m_{\ph}$ on $\K$. Assume, via contradiction, that $a \in \K$ and $\nu(a) < m_{\ph}(a)$. Then $\frac{\ph}{z - a} \in \Hone(\Om)$, so $\frac{\ph}{z - a} \in M(\K,\nu) = \ph \Hone(\Om)$. This implies $\frac{1}{z - a} \in \Hone(\Om)$, an impossibility. Hence $\nu(a) = m_{\ph}(a)$ for every $a \in \K$. A similar argument shows that $\ph(a) \ne 0$ for every $a \in \Om\setminus \K$. Hence $\cZ(\ph) = \K$ and $m_{\ph} = \nu$.
\end{proof}

\begin{corns}
If $0 \ne f \in \Hone(\Om)$, then there is a unique (up to units) factorization of the inner part of $f$ into a product $\ph_{0}\ph_{1}$, where $\ph_{0}\Hone(\Om) = M(\cZ(f),m_{f})$ and $\cZ(\ph_{1}) = \emp$.
\end{corns}

\begin{proof}
If $f = \ph h$ is the inner-outer factorization of $f$, $\cZ(h) = \emp$, so $\cZ(f) = \cZ(\ph)$ and $m_{f}= m_{\ph}$. If $\ph_{0}$ is the unique (up to equivalence) inner function for which 
\begin{equation*}
\ph_{0}\Hone(\Om) = M(\cZ(f),m_{f}) = M(\cZ(\ph),m_{\ph}),
\end{equation*}
we see that there is a unique (up to equivalence) inner function $\ph_{1}$
such that $\ph = \ph_{0} \ph_{1}$.
\end{proof}

We call $\ph_{0}$ in the preceding corollary the \emph{Blaschke factor} of
$\ph$ (or $f$) and $\ph_{1}$ the \emph{singular factor} of $\ph$.
It is clear that $I^{-1}(  \Om )  $ is precisely the set of
functions that are both inner and outer.

Suppose $u,v$ are inner, then $w= \GCD(u,v)$ is the unique, up to a unit
factor, inner function such that%
\begin{equation*}
\left[  uH^{1}\left(  \Omega \right)  +vH^{1}\left(  \Omega \right)  \right]
^{-\left \Vert {}\right \Vert _{1}}=wH^{1}\left(  \Omega \right)  .
\end{equation*}
We write
\begin{equation*}
w= \GCD\left(  u,v\right)  .
\end{equation*}

\begin{lemns}
If $\ph$, $\rho$ and $\psi$ are inner functions and $\GCD(
\ph,\rho )  =1$ and $\ph\mid \rho \psi$, then $\ph \mid \psi$.
\end{lemns}

\begin{proof}
Choose a Royden-inner function $\gm$ such that $\rho \psi=\ph \gm.$
Since $\GCD(  \ph,\rho )  =1$, we can choose sequences $(c_{n}), (  d_{n} )  $ in $\Hone(  \Om)  $
such that $\onenorm {c_{n}\ph + d_{n}\rho - 1} \ra 0.$
Then $\onenorm{ \psi (  c_{n} \ph + d_{n} \rho - 1)  } \ra 0.$  This means $\psi = \lim_{n \ra \infty} \ph (c_{n}\psi + d_{n}\gm )  \in \ph \Hone(  \Om)  .$ Hence $\ph \mid \psi.$
\end{proof}

%%%%%
\section{Multiplier Pairs}

In this section we will show that the multipliers of $\Hal(\Om)$ are the functions in $\Hinf(\Om)$, and that $(\Hal(\Om),H(\Om))$ is a multiplier pair in the sense of \cite{HN}.

\begin{thmns}
Suppose $\al \in \fN$, $\ps \in \Hal(\Om)$, and $\ps \Hal(\Om) \subset \Hal(\Om)$.
Then $\ps \in \Hinf(\Om)$.
\end{thmns}

\begin{proof}
It was observed earlier that $\Hal(\Om)$ is a functional Banach space on $\Om$. It follows from \cite{HN} that every multiplier of $\Hal(\Om)$ is bounded on $\Om$, which means $\ps \in \Hinf(\Om)$.
\end{proof}

We let $H(\Om)$ denote the vector space of analytic
functions on $\Om$, and we give $H(\Om)$ the topology of
pointwise convergence. This makes $H(\Om)$ a Hausdorff
topological vector space. Pointwise multiplication is a bilinear map%
\begin{equation*}
\cdot:\Hal(\Om)  \times \Hal(\Om) \ra H(\Om)
\end{equation*}
that is jointly continuous, since the evaluation maps at points in $\Om$ are
continuous on $\Hal(\Om)$. The constant function $1$ is an identity and multiplication is associative on a triple whenever the factors are all in $\Hal(\Om)$. Moreover, the set $\Hinf(\Om)$ of multipliers is norm dense in $\Hal(\Om)$. It follows that $( \Hal(\Om), H(\Om))$ is a \emph{multiplier pair}
as defined in \cite{HN}. The following is an immediate consequence of Theorem 1 of that paper, where an algebra of operators is called reflexive if it contains every operator whose invariant subspaces include those of the algebra.

\begin{propns}
If $\al \in \fN$, then the algebra $\Hinf(\Om)$, acting as multiplications on $\Hal(\Om)$, is maximal abelian and reflexive.
\end{propns}

\begin{proof}
Reflexivity follows because point evaluations are continuous linear functionals that are eigenvectors of the adjoints of all multiplication operators.
\end{proof}

Suppose that $\rho:\Hinf(\Om) \ra \Hinf(\Om) $ is a unital homomorphism and
$\ph = \rho( z)$. Since $z - \lb$ is invertible in
$\Hinf(\Om) $ whenever $\lb \in \C \setminus \Ombar$, we see that $\ph(\Om)  \subset \Ombar$. It follows from the open mapping theorem that either
$\ph(\Om)  \subset \Om$ or $\ph = \lb_{0}%
\in \Gm$. It follows that for every rational function $f$ with poles
off $\Om$ that
\begin{equation*}
\rho(f) = f\of\ph.
\end{equation*}
It is easily shown that if $\ph = \lb_{0} \in \Gm$, then $\rho$
does not extend from $\Hinf(\Om)$ to a bounded operator
on $\Hal(\Om)$. Thus the only
composition operators on $\Hal(\Om) $
in the multiplier-pair sense \cite{HN} have the form%
\begin{equation*}
C_{\ph}f = f\of\ph
\end{equation*}
for some analytic $\ph:\Om \ra \Om$. We do not know which
$\ph$'s give a bounded operator, but we denote this class by $\cF$.
A \emph{local composition operator} $T$ on $\Hal(\Om) $ is an operator such that, for every $f \in \Hal(\Om) $ there is a $\ph_{f} \in \cF$ such
that $Tf=f \of\ph_{f}$. A \emph{local multiplication operator} on
$\Hal(\Om) $ is an operator $S$ such
that, for every $f \in \Hal(\Om) $
there is a $g_{f} \in \Hinf(\Om)$ such that
$Sf = g_{f}\cdot f$. Here are immediate consequences of $(\Hal(\Om) , H(\Om))  $ being a multiplier pair with multipliers in $\Hinf(\Om)$ (see Theorems  2 and 4 in \cite{HN}).

%%%
\begin{propns}
Suppose $\al \in \fN$. Then
\romannumbering
\begin{enumerate}
\item 
every local composition operator on $\Hal(\Om)$ is a composition operator; and
\item 
every local multiplication operator on $\Hal(\Om)$ is multiplication by some member of $\Hinf(\Om)$.
\end{enumerate}
\end{propns}

%%%%%
\section{Affiliated Operators}

Let $\ph$ be a quotient of functions in $\Hinf(\Gm)$: $\ph = \frac{\ps}{\eta}\frac{u}{v}$, where $\ps$ and $\eta$ are inner and have no nontrivial common inner divisor and $u$ and $v$ are outer in $\Hinf(\Gm)$. If $\cD$ consists of all $f \in \Hal(\Gm)$ such that $\ph f \in \Hal(\Gm)$, then let $M_{\ph}$ be the linear transformation from $\cD$ into $\Hal(\Gm)$ defined by $M_{\ph}f = \ph f$. Then $M_{\ph}$ is a closed operator on $\Hal(\Gm)$ with domain $\cD$, and $M_{\ph}$ commutes with multiplication by every function in $\Hinf(\Gm)$.

We will obtain a more useful form for the graph of $M_{\ph}$ than 
\begin{equation*}
\Graph(M_{\ph}) = \set{(f,\ph f) : f \in \cD}. 
\end{equation*}
The functions $u$ and $v$ have moduli with integrable logarithms, and thus $\log(\abs{u} + \abs{v})$ is also integrable. By taking the harmonic extension of $\log(\abs{u} + \abs{v})$ to $\Om$ and adding an appropriate harmonic unit to it as in Lemma~\ref{lem:Per}, we are able to construct a harmonic conjugate of the sum on $\Om$, and by exponentiating the resulting analytic function, we obtain an outer function $F$ in $\Hinf(\Gm)$ with a boundary function satisfying $(1/C) (\abs{u} + \abs{v}) \leq \abs{F} \leq (1/c) (\abs{u} + \abs{v})$ for constants $c, C > 0$. Thus if $a = u/F$ and $b = v/F$ we obtain a pair of outer functions in $\Hinf(\Gm)$ satisfying 
\begin{equation} \label{eqn:abineq}
c \leq \abs{a} + \abs{b} \leq C 
\end{equation}
and $\ph = \frac{\ps}{\eta}\frac{a}{b}$.

The mapping $\Ph: \Hal(\Gm) \ra \Hal(\Gm) \times \Hal(\Gm)$ defined by 
\begin{equation} \label{eqn:defPh}
\Ph(g) = (\eta b g, \ps a g) 
\end{equation}
has its range in $\Graph(M_{\ph})$, and we will show that it is bounded and invertible when $\Hal(\Gm) \times \Hal(\Gm)$ is given either of the equivalent norms $\al_{2}(f,g) = \al(f) + \al(g)$ or $\al_{2}'(f,g) = \al(\abs{f} + \abs{g})$. Equation \eqref{eqn:abineq} implies that $\Ph$ is a bounded operator relative to $\al_{2}'$ that is bounded from below in the sense that $\al_{2}'(\Ph(g)) \geq c_{1}\al(g)$ for some $c_{1} > 0$ and all $g$ in $\Hal(\Gm)$. This is because both $\ps$ and $\eta$ have absolute values that lie in some interval $[c',C']$ with $c'> 0$ and 
\begin{equation*}
\al_{2}'(\Ph(g)) = \al(\abs{\eta b g} + \abs{\ps a g}) = \al((\abs{\eta b} + \abs{\ps a})\abs{g}) = x\al(g),
\end{equation*}
where $x$ is a number in the interval $[cc',CC']$. Thus the range of $\Ph$ is a closed subset of $\Graph(M_{\ph})$.

To see that $\Graph(M_{\ph})$ is included in the range of $\Ph$, suppose $f \in \cD$, so we also have that $\ph f \in \Hal(\Gm)$. Then both $f$ and $\ph f$ belong to $\Hone(\Gm)$ and 
\begin{equation*}
\Bigl| \frac{f}{b} \Bigr| \leq \frac{1}{c}\Bigl( \frac{\abs{b}+ \abs{a}}{\abs{b}} \abs{f} \Bigr) \leq \frac{1}{c}\abs{f} + \frac{C'}{cc'} \abs{\ph f},
\end{equation*}
which implies that both $f/b \in \Hone(\Gm)$ and $f/b \in \Lal(\Gm,\om)$. By Lemma~\ref{lem:HoneLal}, $f/b \in \Hal(\Gm)$. Also, if $\ph f = h$, then $\ps a f = \eta a h$, which implies $\eta \mid \ps a f$, and, since $\ps$ and $\eta$ have no nontrivial common inner divisor and $a$ is outer, it follows that $\eta \mid f$. Thus $f/(\eta b) \in \Hal(\Gm)$, and consequently if $g = f/(\eta b)$, then, by the definition \eqref{eqn:defPh}, we have $\Ph(g) = (f,\ph f)$.

We may summarize the above as follows.

\begin{propns} \label{prop:graphMph}
If $\ph$ is a quotient of functions in $\Hinf(\Gm)$, then there exist inner functions $\ps$ and $\eta$ with no nontrivial common inner divisors, and there exist outer functions $a$ and $b$ in $\Hinf(\Gm)$ such that $\ph = \frac{\ps}{\eta}\frac{a}{b}$ and such that the mapping $\Ph$ of equation~\eqref{eqn:defPh} is a boundedly invertible mapping of $\Hal(\Gm)$ onto $\Graph(M_{\ph}) = \set{(\eta b g, \ps a g): g \in \Hal(\Gm)}$.
\end{propns}

Suppose that $T$ is a closed operator defined on a subspace $\cD(T)$ of $\Hal(\Gm)$ and into $\Hal(\Gm)$ such that $T$ commutes with every multiplication operator $M_{h}$ with $h \in \Hinf(\Gm)$. By \cite{HLN}, there exists a quotient $\ph$ of $\Hinf(\Gm)$ functions such that $Tf = \ph f$ for all $f \in \cD(T)$. The graph of $T$ is a closed subspace of $\Graph(M_{\ph})$ that is the image under $\Ph$ of a closed subspace $\cM$ of $\Hal(\Gm)$. The commuting of $T$ with all multiplications by members of $\Hinf(\Gm)$ makes $\cM$ invariant under multiplication by all members of $\Hinf(\Gm)$, and hence, by Theorem~\ref{thm:BHLal}, $\cM = \xi \Hal(\Gm)$ for some inner function $\xi$. It follows that $\Graph(T) = \xi \Graph(M_{\ph})$. 

\begin{thmns}
If $T$ is a closed operator defined on a subspace $\cD(T)$ of $\Hal(\Gm)$ and $T$ commutes with all multiplications by members of $\Hinf(\Gm)$, then there exists $\ph = \frac{\ps}{\eta}\frac{a}{b}$ where $\ps$ and $\eta$ are Royden inner functions without a nontrivial common inner divisor and $a$ and $b$ are outer in $\Hinf$, and there exists a Royden inner function $\xi$ such that $\cD(T) = \xi\cD(M_{\ph})$ and $T = M_{\ph}\mid \cD(T)$. If, in addition, $T$ is densely defined, then $\eta = 1$ and $\cD(T) = b\Hal(\Gm)$.
\end{thmns}

\begin{proof}
Only the last assertion remains to be verified. If $T$ satisfies the additional requirement of having a dense domain, then it follows from Proposition~\ref{prop:graphMph} that $\eta$ must be invertible and can therefore be absorbed into $\ps$. Also, from the same proposition and the discussion preceding the statement of the theorem, it must be the case that $\xi$ is invertible, and therefore it is not necessary. Thus $\cD(T) = \cD(M_{\ph}) = b\Hal(\Gm)$.
\end{proof}

In the case of densely defined operators of the preceding theorem, the functions $\ph$ can be written in the form $\ph = \frac{\ps a}{b}$ where $\ps$ is inner and $a$ and $b$ are bounded outer functions. These functions constitute the Smirnov class of $\Om$.

\vspace{.25in}
\noindent {\Large S}CHOOL OF {\Large M}ATHEMATICS AND {\Large I}NFORMATION {\Large S}CIENCE, {\Large S}HAANXI {\Large N}ORMAL \\ {\Large U}NIVERSITY, {\Large X}I'AN, {\Large 710119, C}HINA.

\noindent \emph{E-mail address}: \texttt{yanni.chen@snnu.edu.cn}

\vspace{.125in}
\noindent {\Large D}EPARTMENT OF {\Large M}ATHEMATICS, {\Large U}NIVERSITY OF {\Large N}EW {\Large H}AMPSHIRE, {\Large D}URHAM, {\Large NH 03824, U.S.A}.

\noindent\emph{E-mail address}: \texttt{don@unh.edu}

\vspace{.125in}
\noindent {\Large D}EPARTMENT OF {\Large M}ATHEMATICS, {\Large U}NIVERSITY OF {\Large C}ENTRAL {\Large F}LORIDA, {\Large O}RLANDO, {\Large FL 32816, U.S.A.}

\noindent\emph{E-mail address}: \texttt{zhe.liu@ucf.edu}

\vspace{.125in}
\noindent {\Large D}EPARTMENT OF {\Large M}ATHEMATICS, {\Large U}NIVERSITY OF {\Large N}EW {\Large H}AMPSHIRE, {\Large D}URHAM, {\Large NH 03824, U.S.A}.

\noindent\emph{E-mail address}: \texttt{ean@unh.edu}


\begin{thebibliography}{99}
\bibitem{A}
A. Aleman, private communication.

\bibitem{AR1}
A. Aleman, S. Richter,  
Simply invariant subspaces of $\Htwo$ of some multiply connected regions, 
Integral Equations Operator Theory {\bf 24} (1996), no. 2, 127--155. MR1371943

\bibitem{AR2}
A. Aleman, S. Richter, Erratum: ``Simply invariant subspaces of $\Htwo$ of some multiply connected regions'',
Integral Equations Operator Theory {\bf 29} (1997), no. 4, 501--504. MR1484863

\bibitem{B}
A. Beurling, On two problems concerning linear transformations in Hilbert space, Acta Math. {\bf 81} (1949), 239--255. MR0027954

\bibitem{Chen1}
Y. Chen, Lebesgue and Hardy spaces for symmetric norms I, arXiv:1407.7920 [math.OA] 30 July 2014.

\bibitem{Chen2}
Y. Chen, A general Beurling-Helson-Lowdenslager theorem on the disk, to appear in Adv. in Appl. Math.

\bibitem{C}
J. B. Conway, \emph{Functions of  one complex variable II}, Graduate Texts in Mathematics, 159, Springer-Verlag, New York, 1995. MR1344449

\bibitem{Fisher}
S. D. Fisher, \emph{Function theory on planar domains. A second course in complex analysis}, Pure and Applied Mathematics (New York), A Wiley-Interscience Publication, John Wiley \& Sons, Inc., New York, 1983. MR0694693

\bibitem {Forelli}
F. Forelli, 
Bounded holomorphic functions and projections, 
Illinois J. Math. {\bf 10} (1966), 367--380. MR0193534 (33 \#1754)                       

\bibitem{G1}
T. W. Gamelin, $\Hp$ spaces and extremal functions in $\Hone$, Trans. Amer. Math. Soc. {\bf 124} (1966), 158--167. MR0213877

\bibitem{G2}
T. W. Gamelin, \emph{Uniform algebras}, Prentice-Hall, Inc., Englewood Cliffs, N. J., 1969,. MR0410387

\bibitem{HLN}
D. Hadwin, Z. Liu, E. Nordgren, Closed densely defined operators commuting with multiplications in a multiplier pair, Proc. Amer. Math. Soc. {\bf 141} (2013), no. 9, 3093--3105. MR3068963 

\bibitem{HN}
D. Hadwin, E. Nordgren,  A general view of multipliers and composition operators, Linear Algebra Appl. {\bf 383} (2004), 187--211. MR2073904

\bibitem{H}
M. Hasumi, Invariant subspace theorems for finite Riemann surfaces, Canad. J. Math. {\bf 18} (1966), 240--255. MR0190790

\bibitem {HL}
H. Helson, D. Lowdenslager, Invariant subspaces, 1961  
\emph{Proc. Internat. Sympos. Linear Spaces} (Jerusalem, 1960), 251--262, Jerusalem Academic Press, Jerusalem; Pergamon, Oxford. MR0157251

\bibitem{Hitt}
D. Hitt,
Invariant subspaces of $\cH^{2}$ of an annulus,
Pacific J. Math. {\bf 134} (1988), no. 1, 101--120. MR953502 (90a:46059)

\bibitem{K}
C. Kuratowski, Th\'{e}or\`{e}mes sur l'homotopie des fonctions continues de variable complexe et leurs rapports \`{a} la Th\'{e}orie des fonctions analatiques, Fund. Math. {\bf 33} (1945), 316--341. 

\bibitem{N}
Z. Nehari, 
\emph{Conformal mapping}, 
McGraw-Hill Book Co., Inc., New York, Toronto, London, 1952. viii+396 pp. MR0045823

\bibitem {P}
M. Parreau, Sur les moyennes des fonctions harmoniques et analytiques et la classification des surfaces de Riemann, Ann. Inst. Fourier Grenoble {\bf 3} (1951), 103--197. MR0050023

\bibitem {R}
H. L. Royden, Invariant subspaces of $\cH^{p}$ for multiply connected regions, Pacific J. Math. {\bf 134} (1988), no. 1, 151--172. MR0953505

\bibitem {Ru1}
W. Rudinr, Analytic functions of class $H_{p}$, Trans. Amer. Math. Soc. {\bf 78} (1955), 46--66. MR0067993

\bibitem {Ru2}
W. Rudin, \emph{Real and complex analysis},
Third edition, McGraw-Hill Book Co., New York, 1987. MR0924157

\bibitem{Rudol}
K. Rudol, Some results related to Beurling's theorem, Univ. Iagel. Acta Math., No. 38 (2000), 289--298. MR1812120

\bibitem{SZ}
S. Saks, A. Zygmund, \emph{Analytic Functions}, Third edition. Elsevier Publishing Co., Amsterdam-London-New York; PWN---Polish Scientific Publishers, Warsaw, 1971. MR0055432

\bibitem{Sarason}
D. Sarason,
The $\Hp$ spaces of an annulus, 
Mem. Amer. Math. Soc. No. {\bf 56}, 1965, 78 pp. MR0188824

\bibitem{T}
M. Tsuji, \emph{Potential theory in modern function theory}, Maruzen Co., Ltd., Tokyo, 1959, MR0114894

\bibitem{Voichick64}
M. Voichick, Ideals and invariant subspaces of analytic functions, Trans. Amer. Math. Soc. {\bf 111} (1964), 493--512. MR0160920

\bibitem {Voichick66}
M. Voichick, Invariant subspaces on Riemann surfaces, Canad. J. Math. {\bf 18} (1966), 399--403. MR0190791
\end{thebibliography}
\end{document}